\documentclass[a4paper]{article}
\usepackage[cp1250]{inputenc}
\usepackage[english]{babel}
\usepackage{amssymb}
\usepackage{amsmath}
\usepackage{mathrsfs}
\usepackage{amsfonts}
\usepackage{color}
\usepackage{vmargin}
\usepackage{amsthm}
\usepackage{hyperref}
\usepackage{verbatim}
\usepackage{indentfirst}
\usepackage{appendix}
\usepackage{cite}
\usepackage{mathtools}

\long\def\symbolfootnote[#1]#2{\begingroup%
\def\thefootnote{\fnsymbol{footnote}}\footnote[#1]{#2}\endgroup}

\setmarginsrb{20mm}{20mm}{20mm}{20mm}{10mm}{10mm}{10mm}{10mm}

\newtheoremstyle{remark}
  {}{}{}{}{\bfseries}{.}{.5em}{{\thmname{#1 }}{\thmnumber{#2}}{\thmnote{ (#3)}}}

\newtheorem{theo}{Theorem}[section]
\newtheorem{lem}{Lemma}[section]

\theoremstyle{definition}
\newtheorem*{ex}{Example}
\theoremstyle{remark}

\newcommand{\dx}{ {\rm d}}
\newcommand{\otb}{\overline{\Omega}_T}
\newcommand{\supp}{\textnormal{supp}}
\newcommand{\gt}{\mathcal{G}_T}

\newcommand{\ot}{\Omega_T}
\newcommand{\otc}{\overline{\Omega}_T}
\newcommand{\ots}{\Omega_{T,\sigma}}
\newcommand{\otsc}{\overline{\Omega}_{T,\sigma}}
\newcommand{\chpbt}{C^{\alpha(\cdot)}}
\newcommand{\chp}[2]{C^{#1,#2,\alpha(\cdot)}}
\newcommand{\ac}{\alpha(\cdot)}
\newcommand{\R}{\mathbb{R}}
\newcommand{\expo}[1]{\operatorname{exp}\left[#1\right]}
\newcommand{\dist}[2]{\textnormal{dist}\left(#1,#2\right)}
\newcommand{\diam}[1]{\textnormal{diam}\left(#1\right)}
\newcommand{\otg}{\Omega_T\cup\Gamma}
\newcommand{\I}{\mathcal{I}}
\newcommand{\xfrac}[2]{#1/#2}

\let \epsilon \varepsilon
\let \phi \varphi
\setmarginsrb{20mm}{20mm}{20mm}{20mm}{10mm}{10mm}{10mm}{10mm}

\begin{document}
\title{\bf A priori estimates and theory of existence for parabolic equations in variable H\"older spaces}
\author{ Piotr Micha{\l} Bies\\
\it\small{Department of Mathematics and Information Sciences,}\\
\it\small{Warsaw University of Technology,}\\
\it\small{Ul. Koszykowa 75, 00-662 Warsaw, Poland.}\\
{\tt biesp@mini.pw.edu.pl}}
\maketitle
\begin{abstract}
We study parabolic equations in variable H\"{o}lder spaces on domains of Euclidean spaces.  The existence and uniqueness of solutions is proved.
\end{abstract}


\section{Introduction}

Let $\Omega$ be a bounded domain in $\mathbb{R}^n$ and let $T>0$.
We investigate linear parabolic operators in the following form
\begin{equation*}
u_t-Lu,
\end{equation*}
where the coefficients of the elliptic operator $L$ are in the variable H\"{o}lder space $C^{\alpha(\cdot)}(\otc)$. 

We are interested in the study of the following boundary value problem
\begin{align}\label{diripr}
\begin{cases}
u_t-L u=f & \text{in} \, \ot,\\
u=\phi & \text{ on} \, \partial\Omega\times (0,T),\\
u=g&\text{ on}\,\Omega\times\{0\},
\end{cases}
\end{align}
where $f$, $\phi$ and $g$ are elements of  variable H\"{o}lder spaces. We prove Schauder estimates for this problem with exponents $\alpha$ satisfying the so-called log-H\"{o}lder condition. Furthermore, existence and uniqueness of solutions to problem (\ref{diripr}) is proven in $\chp{2}{1}(\otb)$.

Variable function spaces were introduced as a tool to study partial differential equations with nonstandard growth (see \cite{kov,zhi}). However, these spaces are connected in a natural way with some engineer and computer science problems. Namely, they are used to model fluids which viscosity change in response to an electric field, i.e. \emph{electrorheological fluids} (see \cite{ruz}). In paper \cite{term} appears a model of a thermorheological flow, which is a fluid whose viscosity depends on a temperature. Furthermore, these spaces can be used in the image denoising (see \cite{den,pi}). Blomgren et al. \cite{blom} suggested that in the image processing an image of a better quality can be obtained by an interpolation techniques that use variable exponents.

Linear parabolic equations are natural extension of the elliptic ones. Elliptic equations are considered in \cite{gt, krylov}. We studied elliptic equations in variable H\"older spaces in \cite{bies1,bies}. Parabolic equations with coefficients in other function spaces were studied in the mathematical literature. For instance, the parabolic 
problems in the classical H\"{o}lder spaces are discussed in \cite{par1,par2,par3,par4,par5,par6,par7,par9,par8,par10,  par11,par12,par13,par14}. We strongly recommend monographs devoted to this topic \cite{fried,krylov,lieb}. 

The paper is divided into six sections. In Section \ref{prel} we introduce notations and define H\"older spaces with variable exponent. In Section~\ref{intsec} we prove a priori interior estimates for the heat equation and then for general parabolic equation. 
 Section~\ref{bousec} is devoted to study boundary Schauder estimates. Next, in Section \ref{glosec} we prove global a priori estimates. Finally, in Section \ref{molsec}  we obtain existence and uniqueness of solutions to boundary value problems. Moreover, we present an equation for which variable H\"older spaces are optimal. We formulate the interpolation inequalities in parabolic H\"older spaces in Appendix \ref{apint}. In Appendix \ref{apest} we show an estimation for certain integral operator.
 
\section{Preliminaries}\label{prel}
Let $\Omega\subset\mathbb{R}^n$ be an open and bounded set and let $T>0$. Denote $\Omega_T=\Omega\times(0,T)$. We call  a set $\gt=\partial\ot\setminus\Omega\times\{T\}$ a parabolic boundary of $\ot$. We define a metric  
\begin{align}\label{metr}
d(P,Q)=\max (|x_1-x_2|,\sqrt{|t_1-t_2|})
\end{align}
for $P=(x_1,t_1),Q=(x_2,t_2)\in \ot$. This metric is equivalent to the  Euclidean distance on $\ot$. In this paper the norm $|\cdot|$ for space variables is the maximum norm i.e.
\begin{align*}
|x|=\max_{i=1,\ldots,n}|x_i|\qquad\textnormal{ for }x=(x_1,\ldots,x_n)\in\mathbb{R}^n.
\end{align*}

A set $C({\otc})$ consists of all continuous functions on $\ot$, which can be extended continuously on $\otc$. Let 
$$C^{2,1}(\otc)=\left\{u\in C(\otc)\colon D^{\alpha}u\in C(\otc)\textrm{ and } u_t\in C(\otc) \textrm{ for all }|\alpha|\leq 2\right\}.$$
Higher order spaces are defined similarly; for $k\in\mathbb{N}$ we denote
\begin{align*}
C^k(\otc)=\left\{u\in C(\otc)\colon D^{\beta}u\in C(\otc)\textrm{ for all } |\beta|\leq k\right\}.
\end{align*}

Now, we turn to an introduction of basics of the theory of variable H\"older spaces (for details see \cite{alm2,alm1}). If one wants to read more about function spaces with a variable exponent, we refer to the monographs \cite{uribe} and~\cite{dien}. Any function $\alpha\colon\ot\to(0,1]$ will be called a variable exponent.
A semi-norm for a function $u$ is defined as follows
\begin{align*}
[u]_{\alpha(\cdot),\ot}=[u]_{0,\ac,\ot}=\sup_{\substack{P,Q\in\ot\\P\neq Q}}\frac{|u(P)-u(Q)|}{d^{\alpha(Q)}(P,Q)}.
\end{align*}
Hence, we are able to define a space with variable exponent
\begin{align*}
\chpbt(\otc)=\left\{ u\in C(\otc)\colon [u]_{\ac,\ot}<\infty\right\}
\end{align*}
and
\begin{align*}
\chp{2}{1}(\otc)=\left\{u\in C^{2,1}(\otc)\colon[u_t]_{\ac,\ot}<\infty \textrm{ and } [D^2 u]_{\ac,\ot}<\infty\right\}.
\end{align*}

Analogously we define $C^{k,\ac}(\otc)$ for $k\in\mathbb{N}$.
Let $u$ be a function defined on $\ot$. We introduce the following notations
\begin{align*}
|u|_{0,\ot}=\sup_{X\in\ot}|u(X)|,\qquad|u|_{0,\ac,\ot}=|u|_{0,\ot}+[u]_{0,\ac,\ot}.
\end{align*}
It is easy to see that $C(\otc)$ with $|\cdot|_{0,\ac,\ot}$ is a Banach space. For $u\in C^k(\otc)$ and for $k\in\mathbb{N}$ we introduce 
\begin{align*}
[u]_{k,\ot}=\sup_{X\in\ot}|D^ku(X)|,\qquad
|u|_{k,\ot}=\sum_{i=0}^k[u]_{k,\ot}.
\end{align*}
Finally, for $u\in C^{2,1}(\otc)$ we define
\begin{align*}
|u|_{2,1,\alpha(\cdot),\ot}=|u|_{{0,\ot}}+|Du|_{{0,\ot}}+|D^2u|_{0,\alpha(\cdot),\ot}+|u_t|_{0,\alpha(\cdot),\ot}.
\end{align*}
Once more, $C^{2,1,\ac}(\ot)$ is a Banach space with $|\cdot|_{2,1,\ac,\ot}$.

For a fixed $P\in\ot$ and a function $u$ we define
\begin{align*}
[u]_{\alpha(P),P,\ot}=\sup_{\substack{Q\in\ot\\Q\neq P}}\frac{|u(P)-u(Q)|}{d^{\alpha(P)}(P,Q)}.
\end{align*}
We restrict our attention to certain class of exponents $\alpha$. The exponent $\alpha\colon\ot\to (0,1]$ is called log-H\"older continuous, if for all $P,Q\in\ot$ the inequality
\begin{align}\label{rev1}
|\alpha(P)-\alpha(Q)|\left|\ln d(P,Q)\right|\leq M
\end{align}
holds for some $M>0$. The smallest $M$ satisfying (\ref{rev1}) is denoted by $c_{\log}(\alpha)$. Furthermore, we denote 
\begin{align*}
\alpha^+=\sup_{X\in\ot}\alpha(X),\qquad\alpha^-=\inf_{X\in\ot}\alpha(X).
\end{align*}
We shall consider only exponents which are log-H\"older continuous and satisfy
\begin{align*}
0<\alpha^-\leq\alpha^+<1.
\end{align*}

We define a spacetime semicube with top $P=(x_0,t_0)\in\R^n\times\R_+$ and radius $\delta>0$:
\begin{align*}
N(P,\delta)=\{Q=(x,t)\colon d(P,Q)\leq \delta \textrm{ and } t\leq t_0\}. 
\end{align*}
We remind the fundamental solution of the heat equation
 \begin{align*}
G(x-y,t-s)=G(x,t;y,s)=\frac{(s-t)^{-n/2}}{(2\sqrt{\pi})^n}\expo{-\frac{|x-y|^2}{4(s-t)}}
\end{align*}
for $x,y\in\R^n$ and $s,t\in\R$ such that $s>t$.
It is easy to verify that $G$ satisfies the following equations
\begin{align*}
G_s-\Delta_yG=0
\end{align*}
and
\begin{align*}
G_t+\Delta_xG=0.
\end{align*}
It can be also proved
\begin{align}\label{helpineq2}
|D_s^kD^j_y G(x-y, t-s)|\leq & C(s-t)^{-\frac{(n+2k+j)}{2}}\expo{-\frac{|x-y|^2}{5(s-t)}},
\end{align}
where $0\leq k+j\leq 4$.

In paper \cite{bies1}, the following useful lemma about an extension of functions from H\"older spaces with a variable exponent was proved.
\begin{lem}\label{extlemsch}
Let $\Omega\subset\mathbb{R}^n$ be an open and bounded set with the boundary of class $C^2$ and let $\alpha\in\mathcal{A}^{\log}(\Omega)$. 
 Then, there exists $\sigma>0$ and $\bar{\alpha}\in\mathcal{A}^{\log}({\Omega}_{\sigma})$  with $\bar{\alpha}|_{\Omega}=\alpha$, $\bar{\alpha}^+=\alpha^+$,
 $\bar{\alpha}^-=\alpha^-$ such that for any $f\in C^{\alpha({\cdot})}(\bar{\Omega})$, there exists 
$\bar{f}\in C^{\bar{\alpha}(\cdot)}(\bar{\Omega}_{\sigma})$ for which $\bar{f}|_{\Omega}=f$. Moreover, there exists a constant $C=C(\Omega, n, \alpha^-, \alpha^+, c_{\log}(\alpha))$ such that the inequality
\begin{align*}
|\bar{f}|_{0,\bar{\alpha}(\cdot),\Omega_{\sigma}}\leq C|f|_{0,\alpha(\cdot),\Omega}
\end{align*}
holds.
\end{lem}
In the above Lemma the standard norm in variable H\"older spaces is used i.e.
 \begin{align*}
|u|_{0,\ac,\Omega}=|u|_{0,\Omega}+\sup_{\substack{x,y\in\Omega\\x\neq y}}\frac{|u(x)-u(y)|}{|x-y|^{\alpha(x)}}.
\end{align*}

\section{Interior estimates}\label{intsec}

In this section we prove Schauder estimates. We start with the interior estimates. 
 Let us fix  $P\in \R^{n}\times\R_+$ and $d>0$. We will denote $N=N(P,d)$.
The following theorem is fundamental in this part of the article.
\begin{theo}\label{lem1}
Let $\Omega\subset\R^n$ be an open and bounded set and let $T>0$. Let $N\subset\ot$ and let $u\in \chp{2}{1}(\overline{N})$ be a solution of the equation  
$$
L_0u=\Delta u-u_t=f,
$$
 where $f\in\chpbt(\overline{N})$. Then the inequalities
\begin{align}
|D^2u(P)|&\leq C\left( |f|_{0,N}+d^{\alpha(P)}[f]_{\alpha(P),P,N}+|u|_{0,N}d^{-2}\right),\nonumber\\
d^{\alpha(Q)}\frac{|D^2u(Q)-D^2u(P)|}{d^{\alpha(Q)}(Q,P)}&\leq C\left(|f|_{0,N}+d^{\alpha(P)}[f]_{\alpha(P),P,N}+d^{\alpha(Q)}[f]_{\alpha(Q),Q,N}+|u|_{0,N}d^{-2}\right)
\label{tezin1}\end{align}
hold for any $P$ such that $d(P,Q)\leq \frac{d}{4}$ and $Q\in N$, where
  $C=C\left(\textnormal{diam}\left(\ot\right),n, \alpha^+,\alpha^-,c_{\log}(\alpha) \right)$.
\end{theo}
\begin{proof}
We divide the proof into a few steps. 

\textbf{1.}
For $\eta>0$ set $N_{\eta}= N(P,\eta\delta)$. Let $\phi\in C^3_c(N)$ be a cut-off function such that $\phi(Q)=0$ for $Q\in N\setminus N_{\xfrac{3}{4}}$, $\phi(Q)=1$ for $Q\in N_{\xfrac{1}{2}}$ and such that
\begin{align}\label{helpineq1}
|D^k_xD^h_t\phi (x,t)|\leq Dd^{-k-2h}\textrm{ for }(x,t)\in N.
\end{align}
Let $P=(x_0,t_0+d^2)$ and let us denote $B=\{x\in\mathbb{R}^n\colon |x-x_0|\leq d\}$. Then, we see $N=\overline{B}\times[t_0,t_0+d^2]$. We take $Q=(y,s) \in N$, such that $d(P,Q)\leq \frac{d}{4}$. We set $v(x,t)=\phi(x,t) G(x-y;t-s)$. From \cite{fried} we know the following equality  
\begin{align}\label{green}
vL_0u-uL_0^*v=\sum_{i=1}^nD_i(vu_{x_i}-uv_{x_i}) -(uv)_t
\end{align}
for sufficiently smooth functions $u$ and $v$, where $L_0^*$ is an adjoint operator to $L_0$\footnote{$L_0^*w=\Delta w+w_t$ for an arbitrary function $w$.}.

We integrate  above equality and we obtain
\begin{align*}
\int_{t_0}^s\int_B \phi(x,t) G(x-y;t-s) f(x,t)\dx x\dx t&-\int_{t_0}^s\int_Bu(x,t)L_0^*\left(\phi(x,t) G(x-y;t-s)\right)\dx x\dx t=\\
\int_{t_0}^s\int_B\sum_{i=1}^nD_i(vu_{x_i}-uv_{x_i}) -(uv)_t\dx x\dx t&=\int_{t_0}^s\int_{\partial B}vDu\cdot n-uDv\cdot n\dx S(x)\dx t-u(y,s).
\end{align*}
Therefore, the above equality yields
\begin{align*}
u(y,s)&=-\int_{t_0}^s\int_B \phi(x,t) G(x-y;t-s) f(x,t)\dx x\dx t+\int_{t_0}^s\int_Bu(x,t)L_0^*\left(\phi(x,t) G(x-y;t-s)\right)\dx x\dx t\\
&=-H_0+J_0.
\end{align*}

We set 
\begin{align*}
H(Q)=D^2_{y}H_0,\qquad J(Q)=D^2_y J_0,
\end{align*}
where $Q=(y,s)$. 
Thus, we see that we have
\begin{align*}
D^2u(Q)=-H(Q)+J(Q).
\end{align*}

\textbf{2.} We estimate $|D^2u(Q)|$. First, we will control $J(Q)$.
We denote $U=N_{3/4}\setminus N_{1/2}\cap\{t_0\leq t\leq s\}$ and we have
\begin{align*}
J=\int_{U}u(x,t)(\Delta+D_t)\left(\phi(x,t)D^2_yG(x-y,t-s)\right)\dx x\dx t.
\end{align*}
We know that $(\Delta+D_t)G=0$. Thus, the above equality yields
\begin{align}\label{2505in}
\begin{split}
&J=\\
&\int_{U}u(x,t)\left[\phi_t(x,t)D^2_y G(x-y, t-s)+D^2_x\phi(x,t)D^2_y G(x-y, t-s)+2D_x\phi(x,t)D^3_yD(x-y,t-s)\right]\dx x \dx t.
\end{split}
\end{align}
Next, we use inequalities (\ref{helpineq1}) and (\ref{helpineq2}) and we get
\begin{align}\label{helineqa}
|J|&\leq C |u|_{0,N}\sum_{i=0}^1d^{-2+i}\int_{U}(s-t)^{-(n+2+i)/2}\expo{-\frac{|x-y|^2}{5(s-t)}}\dx x \dx t.
\end{align}

Let us define the subset of $U$
$$
U_1=\left\{(x,t)\in U\colon |x-x_0|>\tfrac{d}{2}\right\}.
$$
We also set $U_2=U\setminus U_1$. 
For $(x,t)\in U_1$ we have 
\begin{align}\label{3105in1}
|x-y|\geq |x-x_0|-|x_0-y|>\tfrac{d}{2}-\tfrac{d}{4}=\tfrac{d}{4}.
\end{align}
We  decompose integrals from (\ref{helineqa}) as follows $\int_{U_1}\ldots +\int_{U_2}\ldots.$
First, we estimate $\int_{U_1}\ldots .$ 
 We use inequality (\ref{3105in1}) and then substitute $z=d^2/(s-t)$. It yields
\begin{multline}\label{in1}
\int_{U_1}(s-t)^{-(n+2+i)/2}\expo{-\frac{|x-y|^2}{5(s-t)}}\dx x \dx t\leq d^n\int_0^{s}(s-t)^{-(n+2+i)/2}\expo{-C\frac{d^2}{(s-t)}}\dx t\\
\leq d^n\int_0^{\infty}\frac{d^{-(n+2+i)}}{z^{-(n+2+i)/2}}\expo{-Cz}\frac{d^2}{z^2}\dx z=\int_0^{\infty}d^{-i}{z^{(n-2+i)/2}}\expo{-Cz}\dx z\leq C d^{-i}.
\end{multline}
It lefts to estimate $\int_{U_2}$.
If $(x,t)\in U_2$, then $(x,t)\notin N_{1/2}$. Hence, $|t-\tilde{t}_0| > (d/2)^2$, where $\tilde{t}_0=t_0+d^2$. Thus, we conclude
\begin{align*}
s-t=s-\tilde{t}_0+\tilde{t}_0-t > \left(\frac{d}{2}\right)^2-\left(\frac{d}{4}\right)^2>\left(\frac{d}{4}\right)^2.
\end{align*}
Thanks to this inequality we have 
\begin{align}\label{in2}
\int_{U_2}(s-t)^{-(n+2+i)/2}\expo{-\frac{|x-y|^2}{5(s-t)}}\dx x \dx t\leq C d^{-(n+2+i)}|\Omega_2|\leq\frac{C}{d^i}.
\end{align}
We put inequalities (\ref{in1}) and (\ref{in2}) to (\ref{helineqa}) and we get
\begin{align}\label{pomin1}
|J|\leq C|u|_{0, N}d^{-2}.
\end{align}

\textbf{3.} Next, we shall estimate $H$
\begin{align}\label{kini}
H=&{}\int_{t_0}^s\int_BD^2_yG(x-y,t-s)\left(\phi(x,t)f(x,t)-\phi(y,s)f(y,s)\right)\dx x\dx t\nonumber\\&+\phi(y,s)f(y,s)\int_{t_0}^s\int_BD^2_yG(x-y,t-s)\dx x\dx t=H_1+H_2.
\end{align}
First, we estimate the term from the integral $H_1$. For $(x,t)\in N$ we have
\begin{align}\label{inlem1}
\begin{split}
|\phi(x,t)f(x,t)&-\phi(y,s)f(y,s)|\leq|\phi(x,t)f(x,t)-\phi(x,t)f(y,s)|+|\phi(x,t)f(y,s)-\phi(y,s)f(y,s)|\\
&\leq C\left[[f]_{\alpha(Q),Q,N}\left(|x-y|^{\alpha(Q)}+|t-s|^{\alpha(Q)/2}\right)+|f|_{0,N}\left(\frac{|x-y|}{d}+\frac{|t-s|}{d^2}\right)\right],
\end{split}
\end{align}
where inequality (\ref{helpineq1}) was applied.
We put the above result into $H_1$ and then we use (\ref{helpineq2}). It yields
\begin{align}\label{intin}
|H_1|\leq {}&C [f]_{\alpha(Q),Q,N}\int_{t_0}^s\int_B(s-t)^{-(n+2)/2}\expo{-\frac{|x-y|^2}{5(s-t)}}\left(|x-y|^{\alpha(Q)}+|t-s|^{\alpha(Q)/2}\right)\dx x\dx t\nonumber
\\
&+C|f|_{0,N}\int_B(s-t)^{-(n+2)/2}\expo{-\frac{|x-y|^2}{5(s-t)}}\left(\frac{|x-y|}{d}+\frac{|t-s|}{d^2}\right)\dx x \dx t.
\end{align}
In the first integral we substitute $|x-y|=r(s-t)^{1/2}$ and then we estimate it as follows
\begin{align*}
\int_{t_0}^s\int_0^{\infty}&\frac{(s-t)^{\alpha(Q)/2}r^{\alpha(Q)}+(s-t)^{\alpha(Q)/2}}{(s-t)^{(n+2)/2}}\expo{-C r}r^{n-1}(s-t)^{n/2}\dx r\dx t\\
&\leq C\int_{t_0}^s\frac{\dx t}{(s-t)^{1-\alpha(Q)/2}}=C\frac{(s-t_0)^{\alpha(Q)/2}}{\alpha(Q)/2}\leq C\frac{(s-t_0)^{\alpha(Q)/2}}{\alpha^-}\leq C d^{\alpha(Q)},
\end{align*}
where in the last inequality we have used the fact that $s-t_0\leq d^2$. 
Now, we estimate the second integral in (\ref{intin}). We again substitute $|x-y|=r(s-t)^{1/2}$ and we obtain
\begin{align*}
C \int_{t_0}^s&\int_B(s-t)^{-(n+2)/2}\expo{-\frac{|x-y|^2}{5(s-t)}}\left(\frac{|x-y|}{d}+\frac{|t-s|}{d^2}\right)\dx x\dx t\\
 &\quad\leq C\int_{t_0}^s\int_0^{\infty}(s-t)^{-(n+2)/2}\expo{-C r}r^{n-1}(s-t)^{n/2}\left(\frac{|s-t|^{1/2}r}{d}+\frac{|t-s|}{d^2}\right)\dx r \dx t
 \\&\quad=C\int_{t_0}^s\int_0^{\infty}(s-t)^{-1/2}d^{-1}\expo{-C r}r^{n}\dx r \dx t+C\int_{t_0}^s\int_0^{\infty}d^{-2}\expo{-C r}r^{n-1}\dx r \dx t\\&\quad
\leq C\left(\int^s_{t_0}\frac{\dx t}{(s-t)^{1/2}d}+\frac{s-t_0}{d^2}\right)=C\left(\frac{2(s-t_0)^{1/2}}{d}+\frac{s-t_0}{d^2}\right)\leq C.
\end{align*}
Hence, we deduce 
\begin{align*}
|H_1|\leq C\left(|f|_{0,N}+d^{\alpha(Q)}[f]_{\alpha(Q),Q,N}\right).
\end{align*}

\textbf{4.} Next, we bound $H_2$. For this purpose we use the Gauss formula
\begin{align}\label{inlem1para}
|H_2|\leq{}& C |f|_{0,N}\int_{t_0}^s\int_{\partial B}|D_yG(x-y,t-s)|\dx S(x)\dx t\nonumber\\
\stackrel{(\ref{helpineq2})}{\leq}{}&C |f|_{0,N}\int_{t_0}^s\int_{\partial B}(s-t)^{-\frac{(n+1)}{2}}\expo{-\frac{|x-y|^2}{5(s-t)}}\dx S(x)\dx t.
\end{align}
Since $|y-x_0|\leq\frac{d}{4}$ and  $|x-x_0|=d$ for $x\in \partial B$, we have
\begin{align*}
|x-y|\leq |x-x_0|-|y-x_0|\leq \tfrac{3}{4}d.
\end{align*}
We put this inequality into (\ref{inlem1para}) and we get
\begin{align}\label{kinii}
|H_2|\leq{}& C |f|_{0,N}\int_{t_0}^s\int_{\partial B}(s-t)^{-{(n+1)}/{2}}\expo{-\frac{Cd^2}{(s-t)}}\dx S(x)\dx t\nonumber\\
={}&C |f|_{0,N}d^{n-1}\int_{t_0}^s(s-t)^{-{(n+1)}/{2}}\expo{-\frac{Cd^2}{(s-t)}}\dx t.
\end{align}
Next, we substitute $z={d^2}/({s-t})$
\begin{align*}
|H_2|\leq{}  C |f|_{0,N}d^{n-1}\int_{\hat{t}}^{\infty}\frac{z^{(n+1)/2}}{d^{n+1}}\frac{d^2}{z^2}\expo{-{Cz}}\dx z\leq C |f|_{0,N},
\end{align*}
where $\hat{t}=d^2/(s-t_0)\geq 1$. Finally, we conclude
\begin{align}\label{pomin2}
|H|\leq C\left(|f|_{0,N}+d^{\alpha(Q)}[f]_{\alpha(Q),Q,N}\right).
\end{align}
Now, we join together (\ref{pomin1}) and (\ref{pomin2}), what yields
\begin{align*}
|D^2_yu(y,s)|\leq C\left( |f|_{0,N}+d^{\alpha(Q)}[f]_{\alpha(Q),Q,N}+|u|_{0,N}d^{-2}\right).
\end{align*}

\textbf{5.} Next, we will estimate the H\"older semi-norm of $D^2_yu(y,s)$ 
\begin{align}\label{lem1in1pom}
\begin{split}
d^{\alpha(Q)}\frac{|D^2_yu(Q)-D^2_yu(P)|}{d^{\alpha(Q)}(Q,P)}&\leq d^{\alpha(Q)}\frac{|H(Q)-H(P)|}{d^{\alpha(Q)}(Q,P)}+d^{\alpha(Q)}\frac{|J(Q)-J(P)|}{d^{\alpha(Q)}(Q,P)}\\
&=H'+J',
\end{split}
\end{align}
where $H$ and $J$ are given as in the previous part of the proof. 
Let us introduce notations $P=(y_1, s_1)$ and $Q=(y_2,s_2)$. We estimate
\begin{align}\label{helllineq}
&|J'|\leq\frac{d^{\alpha(Q)}}{d^{\alpha(Q)}(Q,P)}\left|\int_{t_0}^{s_1}\int_{B}u(x,t)(\Delta+D_t)\left(\phi(x,t)D^2_yG(x-y_1,t-s_1)\right)\dx x\dx t\right.\nonumber\\&\qquad\qquad\qquad
-\left.\int_{t_0}^{s_2}\int_{B}u(x,t)(\Delta+D_t)\left(\phi(x,t)D^2_yG(x-y_2,t-s_2)\right)\dx x\dx t\right|
\leq \frac{d^{\alpha(Q)}}{d^{\alpha(Q)}(Q,P)} |u|_{0,N}\nonumber\\&\qquad\qquad\qquad\qquad
\cdot\left[\sum_{i=0}^1|D^{2-i}_x\phi|_{0,N}\int_{\widetilde{U}_1}\left|D^{2+i}_xG(x-y_1,t-s_1)-D^{2+i}_xG(x-y_2,t-s_2)\right|\dx x\dx t\right.\nonumber\\&\qquad\qquad\qquad
+|D_t\phi|_{0,N}\int_{\widehat{U}_1}|D^{2}_xG(x-y_1,t-s_1)-D^{2}_xG(x-y_2,t-s_2)|\dx x\dx t\nonumber\\&\qquad\qquad\qquad
\left.+\sum_{i=0}^1|D^{2-i}_x\phi|_{0,N}\int_{\widetilde{U}_2}|D^{2+i}_xG(x-y_1,t-s_1)|\dx x\dx t
+|D_t\phi|_{0,N}\int_{\widetilde{U}_2}|D^{2}_xG(x-y_1,t-s_1)|\dx x\dx t\right]\nonumber\\&\qquad
=\frac{d^{\alpha(Q)}}{d^{\alpha(Q)}(Q,P)}|u|_{0,N}\left[\sum_{i=0}^1|D^{2-i}_x\phi|_{0,N}J'_1(i)+|D_t\phi|_{0,N}J'_1(0)+\sum_{i=0}^1|D^{2-i}_x\phi|_{0,N}J'_2(i)+|D_t\phi|_{0,N}J'_2(0)\right],
\end{align}
where $\widetilde{U}_1=N_{1/2}\setminus N_{3/4}\cap \{t_0\leq t\leq s_2\}$ and $\widetilde{U}_2=N_{1/2}\setminus N_{3/4}\cap \{s_2\leq t\leq s_1\}$. We apply the Mean Value Theorem to estimate the first type of integral
\begin{align*}
J'_1(i)\leq\int_{\widetilde{U}_1}|D^{2+i+1}_x G(x-\bar{y}, t-\bar{s})||y_1-y_2|+|D^{2+i}_xD_tG(x-\bar{y},t-\bar{s})||s_1-s_2|\dx x\dx t,
\end{align*}
where point $(\bar{y},\bar{s})$ is on an interval connecting points $(y_1,s_1)$ with $(y_2,s_2)$.

Next, we use  inequality (\ref{helpineq2}) and we obtain
\begin{align}\label{hellineg}
J'_1(i)\leq C\int_{\widetilde{U}_1}C\expo{\frac{-|x-\bar{y}|^2}{5(\bar{s}-t)}}\left((\bar{s}-t)^{-(n+3+i)/2}|y_1-y_2|+(\bar{s}-t)^{-(n+4+i)/2}|s_1-s_2|\right)\dx x\dx t.
\end{align}
Further, we estimate this integral similarly as the integral in (\ref{helineqa}). Therefore, we decompose $\widetilde{U}_1$ into two sets. We define 
$$\widetilde{U}_{11}=\left\{(x,t)\in\widetilde{U}_1\colon |x-x_0|>\tfrac{d}{2}\right\}$$ 
and $$\widetilde{U}_{12}=\widetilde{U}_1\setminus\widetilde{U}_{11}.$$
The integral in (\ref{hellineg}) we divide into $\int_{\widetilde{U}_{11}}\ldots+\int_{\widetilde{U}_{12}}\ldots$.
For $(x,t)\in\widetilde{U}_{11}$ inequality $|x-\bar{y}|\geq\frac{d}{4}$ is satisfied. Thus, we estimate the first part of $J_1'(i)$ as follows
\begin{align*}
\int_{\widetilde{U}_1}C\expo{-\frac{|x-\bar{y}|^2}{5(\bar{s}-t)}}\left((\bar{s}-t)^{-(n+3+i)/2}|y_1-y_2|+(\bar{s}-t)^{-(n+4+i)/2}\right)|s_1-s_2|\dx x\dx t\leq \\ 
\int_{\widetilde{U}_1}C\expo{-\frac{Cd^2}{\bar{s}-t}}\left((\bar{s}-t)^{-(n+3+i)/2}|y_1-y_2|+(\bar{s}-t)^{-(n+4+i)/2}|s_1-s_2|\right)\dx x\dx t.
\end{align*}
In the last integral we substitute $z=d^2/(\bar{s}-t)$ and we obtain that it can be estimated as follows
\begin{align*}
C d^n\int_0^{\infty}&\expo{-C z}\frac{d^2}{z^2}\left(\frac{z^{(n+3+i)/2}}{d^{n+3+i}}|y_1-y_2|+\frac{z^{(n+4+i)/2}}{d^{n+4+i}}|s_1-s_2|\right)\dx z\\
&\leq C\left(\frac{|y_1-y_2|}{d^{1+i}}+\frac{|s_1-s_2|}{d^{2+i}}\right).
\end{align*}
Now, we estimate the integral on $\widetilde{U}_{12}$. There we use inequality $|\bar{s}-t|>({d}/{4})^2$, which is satisfied for $(x,t)\in\widetilde{U}_{12}$. 
Then, we bound $J'_1(i)$ on $\widetilde{U}_{12}$ in the following way 
\begin{align*}
&C(s_2-t_0)|\widetilde{U}_{12}|\left(d^{-(n+3+i)}|y_1-y_2|+d^{-(n+4+i)}|s_1-s_2|\right) \leq\\
 &\quad C d^{n+2}\left(d^{-(n+3+i)}|y_1-y_2|+d^{-(n+4+i)}|s_1-s_2|\right)=C\left(\frac{|y_1-y_2|}{d^{1+i}}+\frac{|s_1-s_2|}{d^{2+i}}\right).
\end{align*}
Finally, we obtain  
\begin{align*}
J_1'(i)\leq C\left(\frac{|y_1-y_2|}{d^{1+i}}+\frac{|s_1-s_2|}{d^{2+i}}\right).
\end{align*}
We see that $|y_1-y_2|\leq\frac{d}{4}$ and $|s_1-s_2|\leq\frac{d^2}{16}$. Thus, we estimate $J'_1(i)$ as follows
\begin{align}\label{helllineq1}
J_1'(i)\leq \frac{C}{d^i}\left(\frac{|y_1-y_2|^{\alpha(Q)}}{d^{\alpha(Q)}}+\frac{|s_1-s_2|^{\alpha(Q)/2}}{d^{\alpha(Q)/2}}\right) \leq \frac{C}{d^i} \frac{d^{\alpha(Q)}(P,Q)}{d^{\alpha(Q)}},
\end{align}
where we have also used the fact that $d(P,Q)\leq d$.

It is left to estimate $J_2'(i)$.
We have $|x-y_1|>d/2$ for $x\in\widetilde{U}_2$. Thus, by (\ref{helpineq2}) we obtain
\begin{align*}
J_2'(i)&\leq C\int^{s_1}_{s_2}\int_{\widetilde{U}_2}(s_1-t)^{-(n+2+i)/2}\expo{-\frac{|y_1-x|^2}{5(s_1-t)}}\dx x\dx t\\
&\leq Cd^n\int_{s_2}^{s_1}(s_1-t)^{-(n+2+i)/2}\expo{-\frac{Cd^2}{s_1-t}}\dx t.
\end{align*}
Then, we substitute $z=d^2/(s_1-t)$
\begin{align*}
J_2'(i)\leq C d^n\int_{\hat{t}_0}^{\infty}\frac{z^{(n+2+i)/2}}{d^{n+2+i}}\frac{d^2}{z^2}\expo{-C z}\dx z\leq \frac{C}{d^i\hat{t}_0^{\mu}},
\end{align*}
where $\hat{t}_0={d^2}/({s_1-s_2})$ and $\mu>0$ is arbitrary constant. Let us take $\mu=\alpha(Q)/2$. We have that $(s_1-s_2)\leq d^2(P,Q)$, so we get
\begin{align}\label{helllineq2}
J_2'(i)\leq\frac{C}{d^i} \frac{d^{\alpha(Q)}(P,Q)}{d^{\alpha(Q)}}.
\end{align}
Finally, we put (\ref{helllineq1}) and (\ref{helllineq2}) into (\ref{helllineq}) and also use (\ref{helpineq1}) and we get
\begin{align}\label{konj}
J'\leq C |u|_{0,N}d^{-2}.
\end{align}

\textbf{6.} The term $H'$ we  estimate as follows
\begin{align}\label{kon}
H'\leq H'_1+H'_2+H'_3+H'_4,
\end{align}
where 
\begin{align*}
H'_1=\frac{d^{\alpha(Q)}}{d^{\alpha(Q)}(P,Q)}&\bigg|\int_{t_0}^{s_2}\int_{B}D^{2}_x(x-y_1,t-s_1)\left(\phi(x,t)f(x,t)-\phi(y_1,s_1)f(s_1,y_1)\right)\\
&-D^2_xG(x-y_2,t-s_2)\left(\phi(x,t)f(x,t)-\phi(y_2,s_2)f(y_2,s_2)\right)\dx x\dx t\bigg|,\\
H'_2=\frac{d^{\alpha(Q)}}{d^{\alpha(Q)}(P,Q)}&\Bigl|\phi(y_1,s_1)f(y_1,s_1)\int_{t_0}^{s_2}\int_BD_x^2G(x-y_1, t-s_1)\dx x\dx t\\
&-\phi(y_2,s_2)f(y_2,s_2)\int_{t_0}^{s_2}\int_BD^2_xG(x-y_2,t-s_2)dx\,dt\Bigr|,\\
H'_3=\frac{d^{\alpha(Q)}}{d^{\alpha(Q)}(P,Q)}&\Bigl|\int_{s_2}^{s_1}\int_BD^2_xG(x-y_1,t-s_1)\left(\phi(x,t)f(x,t)-\phi(y_1,s_1)f(y_1,s_1)\right)\dx x\dx t\Bigr|,\\
H'_4=\frac{d^{\alpha(Q)}}{d^{\alpha(Q)}(P,Q)}&\Bigl|\phi(y_1,s_1)f(y_1,s_1)\int_{s_2}^{s_1}\int_BD^2_xG(x-y_1,t-s_1)\dx x\dx t\Bigr|.
\end{align*}
First, we will bound $H'_3$. We use inequalities (\ref{helpineq1}) and (\ref{inlem1}) and we get
\begin{align*}
H'_3&\leq C \frac{d^{\alpha(Q)}}{d^{\alpha(Q)}(P,Q)}\int_{s_2}^{s_1}\int_B (s_1-t)^{-(n+2)/2}\expo{-\frac{|x-y_1|^2}{5(s_1-t)}}\cdot\\
&\quad\qquad\left[[f]_{\alpha(P),P,N}\left(|x-y_1|^{\alpha(P)}+|t-s_1|^{\alpha(P)/2}\right)+|f|_{0,N}\left(\frac{|x-y_1|}{d}+\frac{|t-s_1|}{d^2}\right)\right]\dx x\dx t.
\end{align*}
We estimate the expression in the similar way as in (\ref{intin}). Thus, we obtain
\begin{align}\label{3106in}
H'_3&\leq C\frac{d^{\alpha(Q)}}{d^{\alpha(Q)}(P,Q)}\left([f]_{\alpha(P),P,N}(s_1-s_2)^{(\alpha(P)/2)} +|f|_{0,N}\left(\frac{2(s_1-s_2)^{1/2}}{d}+\frac{s_1-s_2}{d^2} \right)\right)\nonumber\\
&\leq C\frac{d^{\alpha(Q)}}{d^{\alpha(Q)}(P,Q)}\left([f]_{\alpha(P),P,N}d^{\alpha(P)}(P,Q) +|f|_{0,N}\frac{d^{\alpha(Q)}(P,Q)}{d^{\alpha(Q)}}\right)\nonumber\\
&\leq C\left([f]_{\alpha(P),P,N}d^{\alpha(Q)}+|f|_{0,N}\right),
\end{align}
where we have used the inequality $(s_1-s_1)^{(1/2)}\leq d(P,Q)$ and the fact that $\alpha$ is log-H\"older continuous.

It can be shown that there exist $C_1,C_2$ such that
\begin{align*}
C_1 d^{\alpha(Q)}\leq d^{\alpha(P)}\leq C_2d^{\alpha(Q)},
\end{align*}
where $C_1,\ C_2$ do not depend on $P$ nor $Q$ (for details see the end of the proof of Lemma 3.1 in \cite{bies}).
Thus, from~(\ref{3106in}) we get 
\begin{align}\label{kon1}
H'_3\leq  C\left([f]_{\alpha(P),P,N}d^{\alpha(P)}+|f|_{0,N}\right).
\end{align}
Next, we shall estimate $H'_4$. There we use the Gauss formula
\begin{align*}
H'_4&\leq C\frac{d^{\alpha(Q)}}{d^{\alpha(Q)}(P,Q)}|f|_{0,N}\int_{s_2}^{s_1}\int_{\partial B}\left|D_xG(x-y_1,t-s_1)\right|\dx S\,\dx t\\
&\stackrel{(\ref{helpineq2})}{\leq}C\frac{d^{\alpha(Q)}}{d^{\alpha(Q)}(P,Q)}|f|_{0,N}\int_{s_2}^{s_1}\int_{\partial B}(s_1-t)^{-(n+1)/2}\expo{-\frac{|x-y_1|^2}{5(s_1-t)}}\dx S\dx t.
\end{align*}
Then, we substitute $z=d^2/(s_1-t)$ into the last integral. We use the inequality $|x-y_1|\geq d$ which  is true for $x\in\partial B$
\begin{align*}
H'_4\leq{}&C\frac{d^{\alpha(Q)}}{d^{\alpha(Q)}(P,Q)}|f|_{0,N}d^{n-1}\int^{\infty}_{\hat{t}_0}\frac{z^{(n+1)/2}}{d^{n+1}}\expo{-Cz}\frac{d^2}{z^2}\dx z\\
={}&C\frac{d^{\alpha(Q)}}{d^{\alpha(Q)}(P,Q)}|f|_{0,N}\int^{\infty}_{\hat{t}_0}z^{(n-3)/2}\expo{-Cz}\dx z,
\end{align*}
where $\hat{t}_0={d^2}/({s_1-s_2})\geq 1$. Finally, we have 
\begin{align}\label{kon2}
H'_4\leq&C\frac{d^{\alpha(Q)}}{d^{\alpha(Q)}(P,Q)}|f|_{0,N}\hat{t}_0^{\alpha(Q)/2}\leq C |f|_{0,N}.
\end{align}

Now, we shall estimate $H'_2$. We see that $\phi(y_1,s_1)=\phi(y_2,s_2)=1$, so it yields
\begin{align}\label{in14031}
H'_2\leq {}&\frac{d^{\alpha(Q)}}{d^{\alpha(Q)}(P,Q)}|f|_{0,N}\Bigl|\int_{t_0}^{s_2}\int_BD_x^2G(x-y_1, t-s_1)-D_x^2G(x-y_2, t-s_2)\dx x\dx t\Bigr|\nonumber\\
&+\frac{d^{\alpha(Q)}}{d^{\alpha(Q)}(P,Q)}[f]_{\alpha(P),P,N}d^{\alpha(P)}(P,Q)\Bigl|\int_{t_0}^{s_2}\int_BD^2_x(x-y_1,t-s_1)\dx x \dx t\Bigr|=H'_{21}+H'_{22}.
\end{align}
The term $H'_{22}$ we estimate similarly as the term $H_2$ in   inequality (\ref{kini}) (see (\ref{kinii})). Thus, we obtain
\begin{align}\label{in14032}
H'_{22}\leq Cd^{\alpha(Q)}[f]_{\alpha(P),P,N}\leq Cd^{\alpha(P)}[f]_{\alpha(P),P,N}.
\end{align}
Next, we will estimate the term $H'_{21}$. We use the Gauss formula and the Mean Value Theorem. It yields
\begin{align*}
H'_{21}\leq {}&\frac{d^{\alpha(Q)}}{d^{\alpha(Q)}(P,Q)}|f|_{0,N}\int_{t_0}^{s_2}\int_{\partial B}|D_x^2G(x-\tilde{y}, t-\tilde{s})||y_1-y_2|+|D_xD_tG(x-\tilde{y}, t-\tilde{s})||s_1-s_2|\dx x\dx t,
\end{align*}
where $(\tilde{y},\tilde{s})$ is a point on the segment that joins $P$ with $Q$. We use (\ref{helpineq2}) to the above integral
\begin{multline*}
H'_{21}\leq C\frac{d^{\alpha(Q)}}{d^{\alpha(Q)}(P,Q)}|f|_{0,N}\int_{t_0}^{s_2}\int_{\partial B}\expo{-\frac{|x-\tilde{y}|^2}{5(\tilde{s}-t)}}
\\\cdot\left((\tilde{s}-t)^{-(n+2)/2}|y_1-y_2|+(\tilde{s}-t)^{-(n+3)/2}|s_1-s_2|\right)\dx x\dx t.
\end{multline*}
For $x\in\partial B$ we have  the following inequality
\begin{align*}
|x-\tilde{y}|\geq |x-y_1|-|y_1-\tilde{y}|\geq\frac{3}{4}d.
\end{align*}
It yields
\begin{align*}
H'_{21}\leq C\frac{d^{\alpha(Q)}}{d^{\alpha(Q)}(P,Q)}|f|_{0,N}d^{n-1}\int_{t_0}^{s_2}&\expo{-C\frac{d^2}{\tilde{s}-t}}\\&\cdot\left((\tilde{s}-t)^{-(n+2)/2}|y_1-y_2|+(\tilde{s}-t)^{-(n+3)/2}|s_1-s_2|\right)\dx t.
\end{align*}
We substitute $z={d^2}/({\tilde{s}-t})$
\begin{align}\label{in14033}
H'_{21}\leq {}&C\frac{d^{\alpha(Q)}}{d^{\alpha(Q)}(P,Q)}|f|_{0,N}d^{n-1}\int_{\tilde{t}}^{\infty}\expo{-C z}\left(\frac{z^{(n+2)/2}|y_1-y_2|}{d^{n+2}}+\frac{z^{(n+3)/2}|s_1-s_2|}{d^{n+3}}\right)\frac{d^2}{z^2}\dx z\nonumber\\
={}&C\frac{d^{\alpha(Q)}}{d^{\alpha(Q)}(P,Q)}|f|_{0,N}\int_{\tilde{t}}^{\infty}\expo{-C z}\left(\frac{z^{(n-2)/2}|y_1-y_2|}{d}+\frac{z^{(n-1)/2}|s_1-s_2|}{d^2}\right)\dx z\nonumber\\
\leq{}&C\frac{d^{\alpha(Q)}}{d^{\alpha(Q)}(P,Q)}|f|_{0,N}\left(\frac{d(P,Q)}{d}+\frac{d^2(P,Q)}{d^2}\right)\leq |f|_{0,N},
\end{align}
where $\tilde{t}=d^2/(s_2-t_0)$. 
Now, we put inequalities (\ref{in14032}) and (\ref{in14033}) into (\ref{in14031}) and we get
\begin{align}\label{kon3}
H'_2\leq C\left(d^{\alpha(P)}[f]_{\alpha(P),P,N}+|f|_{0,N}\right).
\end{align}

\textbf{7.} It is left to estimate $H'_1$.
We split $H'_1$ into four terms
\begin{align*}
H'_1\leq C\frac{d^{\alpha(Q)}}{d^{\alpha(Q)}(P,Q)}\left( H'_{11}+H'_{12}+H'_{13}+H'_{14}\right),
\end{align*}
where
\begin{align*}
H'_{11}&=\left|\int_{t_0}^{s_2-\eta}\int_BD^2_xG(x-y_1,t-s_1)-D^2_xG(x-y_2,t-s_2)\right.\\
&\qquad\cdot\left.\left[[f]_{\alpha(P),P,N}\left(|x-y_1|^{\alpha(P)}+|t-s_1|^{\alpha(P)/2}\right)+|f|_{0,N}\left(\frac{|x-y_1|}{d}+\frac{|t-s_1|}{d^2}\right)\right]\dx x\dx y\right|,\\
H'_{12}&=\int_{t_0}^{s_2-\eta}\left|\int_BD^2_xG(x-y_2,t-s_2)\right|\dx x\dx y\\
&\qquad\cdot\left[[f]_{\alpha(P),P,N}\left(|y_2-y_1|^{\alpha(P)}+|s_2-s_1|^{\alpha(P)/2}\right)+|f|_{0,N}\left(\frac{|y_2-y_1|}{d}+\frac{|s_2-s_1|}{d^2}\right)\right],\\
H'_{13}&=\int^{s_2}_{s_2-\eta}\int_B\left|D^2_xG(x-y_1,t-s_1)\right|\\&\qquad\cdot\left[[f]_{\alpha(P),P,N}\left(|x-y_1|^{\alpha(P)}+|t-s_1|^{\alpha(P)/2}\right)+|f|_{0,N}\left(\frac{|x-y_1|}{d}+\frac{|t-s_1|}{d^2}\right)\right]\dx x\dx y,\\
H'_{14}&=\int^{s_2}_{s_2-\eta}\int_B\left|D^2_xG(x-y_2,t-s_2)\right|\\&\qquad\cdot\left[[f]_{\alpha(Q),Q,N}\left(|x-y_2|^{\alpha(Q)}+|t-s_2|^{\alpha(Q)/2}\right)+|f|_{0,N}\left(\frac{|x-y_2|}{d}+\frac{|t-s_2|}{d^2}\right)\right]\dx x\dx y
\end{align*}
and $\eta={d^2(P,Q)}/{4}$.

We use inequality (\ref{helpineq2}) and substitute $|x-y_1|=(s_1-t)^{1/2}\rho$ into the integral $H'_{13}$. The integral with $[f]_{\alpha(P),P,N}$ we estimate in the subsequent way
\begin{align*}
 C \int_{s_2-\eta}^{s_2}&\int_0^{\infty}\expo{-\frac{\rho^2}{5}}\rho^{n-1}\frac{\left((s_1-t)^{\alpha(P)/2}\rho^{\alpha{(P)}}+(s_1-t)^{\alpha(P)/2}\right)}{s_1-t}\dx \rho\dx t \\
&\leq C\left( \int_{s_2-\eta}^{s_2}\frac{(s_1-t)^{\alpha(P)/2}}{s_1-t}\dx t\right)\\
&\leq C\frac{2}{\alpha(P)}((s_1-s_2)^{\alpha(P)/2}-(s_1-s_2+\eta)^{\alpha(P)/2})\leq C\eta^{\alpha(P)/2}.
\end{align*}
The term of $|f|_{0,N}$ we estimate as follows
\begin{align*}
\int_{s_2-\eta}^{s_2}\int_0^{\infty}\expo{-\frac{\rho^2}{5}}\rho^{n-1}\left(\frac{|t-s_1|^{-1/2}}{d}+\frac{1}{d^2}\right)\dx \rho \dx t\leq C\left(\frac{\eta^{1/2}}{d}+\frac{\eta}{d^2}\right).
\end{align*}
Thus, finally we have
\begin{align*}
H'_{13}\leq C\left(\eta^{\alpha(P)/2}[f]_{\alpha(P),P,N}+|f|_{0,N}\left(\frac{\eta^{1/2}}{d}+\frac{\eta}{d^2}\right)\right).
\end{align*}

In similar way we estimate $H'_{14}$ and we have
\begin{align*}
H'_{14}\leq C\left(\eta^{\alpha(P)/2}[f]_{\alpha(P),P,N}+|f|_{0,N}\left(\frac{\eta^{1/2}}{d}+\frac{\eta}{d^2}\right)\right).
\end{align*}
Then, we shall bound the integral $H'_{12}$. We use the Gauss formula and proceed in similar way as in   inequality~(\ref{inlem1para}):
\begin{align*}
H'_{12}\leq C\left( [f]_{\alpha(P),P,N}d^{\alpha(P)}(P,Q)+|f|_{0,N}\left(\frac{d(P,Q)}{d}+\frac{d^2(P,Q)}{d^2}\right)\right).
\end{align*}

It remains to estimate the expression  $H'_{11}$. First of all, we consider the  case  $s_1=s_2$. We estimate the coefficient of $[f]_{\alpha(P),P,N}$
\begin{align*}
\int_{t_0}^{s_2-\eta}&\int_B\int_0^1\left|D^3_xG(x-y_1+z(y_1-y_2), t-s_1)\right||y_1-y_2|\left(|x-y_1|^{\alpha(P)}+|t-s_1|^{\alpha(P)/2}\right)\dx z \dx x\dx t\\&
\leq \int_0^1\int_{t_0}^{s_2-\eta}\int_B\left|D^3_xG(x-y_1+z(y_1-y_2), t-s_1)\right||y_1-y_2|\\
&\qquad\qquad\cdot\left(|x-y_1+z(y_1-y_2)|^{\alpha(P)}+|t-s_1|^{\alpha(P)/2}+|y_1-y_2|^{\alpha(P)}\right)\dx x\dx t\dx z.
\end{align*}
Now, we substitute $|x-y_1+z(y_1-y_2)|=(s_1-t)^{1/2}\rho$ and we deduce
\begin{align*}
C\int_{t_0}^{s_2-\eta}&(s_1-t)^{-3/2}|y_1-y_2|\left(|t-s_1|^{\alpha(P)/2}+|y_1-y_2|^{\alpha(P)}\right)\dx t\\&
\leq C\left(|y_1-y_2|\left(\eta^{(\alpha(P)-1)/2}-(s_1-t_0)^{(\alpha(P)-1)/2})\right)+|y_1-y_2|^{\alpha(P)+1}\left(\eta^{-1/2}-(s_1-t_0)^{-1/2}\right)\right)\\
&\leq C\eta^{\alpha(P)/2}.
\end{align*}
Next, we bound the term multiplied by $|f|_{0,N}$
\begin{align*}
\int_{t_0}^{s_2-\eta}&\int_B\int_0^1\left|D^3_xG(x-y_1+z(y_1-y_2), t-s_1)\right||y_1-y_2|\left(\frac{|x-y_1|}{d}+\frac{|t-s_1|}{d^2}\right)\dx z \dx x\dx t\\&
\leq \int_{t_0}^{s_2-\eta}\int_B\int_0^1\left|D^3_xG(x-y_1+z(y_1-y_2), t-s_1)\right||y_1-y_2|\\
&\qquad\cdot\left(\frac{|x-y_1+z(y_1-y_2)|+|y_1-y_2|}{d}+\frac{|t-s_1|}{d^2}\right)\dx z \dx x\dx t\\&
\leq C\int_{t_0}^{s_2-\eta}(s_1-t)^{-3/2}|y_1-y_2|
`\left(\frac{|t-s_1|^{1/2}+|y_1-y_2|}{d}+\frac{|t-s_1|}{d^2}\right)\dx t
\\&\leq C \left(\frac{\eta}{d}+\frac{\eta^{1/2}}{d^2}\right).
\end{align*}
Let us consider the case $y_1=y_2$. First, we shall estimate the coefficient of $[f]_{\alpha(P),P,N}$
\begin{align*}
\int_{t_0}^{s_2-\eta}\int_B\int_{s_2}^{s_1}&|D^2xD_tG(x-y_1,t-z)|\left(|x-y_1|^{\alpha(P)}+|t-z|^{\alpha(P)/2}+|z-s_1|^{\alpha(P)/2}\right)\dx z \dx x\dx t\\&
\leq C\int_{t_0}^{s_2-\eta}\int_{s_2}^{s_1}(z-t)^{-2}\left((z-t)^{\alpha(P)/2}+(s_1-z)^{\alpha(P)/2}\right)\dx z\dx t\\&
\leq C\eta^{\alpha(P)/2}.
\end{align*}
Finally, we consider the coefficient of $|f|_{0,N}$
\begin{align*}
\int_{t_0}^{s_2-\eta}&\int_B\int_{s_2}^{s_1}|D^2xD_tG(x-y_1,t-z)|\left(\frac{|x-y_1|}{d}+\frac{|t-z|+|z-s_1|}{d^2}\right)\dx z\dx t\\&
\leq C\int_{t_0}^{s_2-\eta}\int_{s_2}^{s_1}(z-t)^{-2}\left(\frac{(z-t)^{1/2}}{d}+\frac{(z-t)+(s_1-z)}{d^2}\right)\dx z \dx t\leq C\left(\frac{\eta^{1/2}}{d}+\frac{\eta}{d^2}\right).
\end{align*}

In general case we decompose $H'_{11}$ as follows
\begin{align*}
H'_{11}\leq&\left|\int_{t_0}^{s_2-\eta}\int_BD^2_xG(x-y_1,t-s_1)-D^2_xG(x-y_1,t-s_2)\right.\\
&\qquad\left.\cdot\left[[f]_{\alpha(P),P,N}\left(|x-y_1|^{\alpha(P)}+|t-s_1|^{\alpha(P)/2}\right)+|f|_{0,N}\left(\frac{|x-y_1|}{d}+\frac{|t-s_1|}{d^2}\right)\right]\dx x\dx y\right|
\\&+\left|\int_{t_0}^{s_2-\eta}\int_BD^2_xG(x-y_1,t-s_2)-D^2_xG(x-y_2,t-s_2)\right.\\
&\qquad\left.\cdot\left[[f]_{\alpha(P),P,N}\left(|x-y_2|^{\alpha(P)}+|t-s_2|^{\alpha(P)/2}\right)+|f|_{0,N}\left(\frac{|x-y_2|}{d}+\frac{|t-s_2|}{d^2}\right)\right]\dx x\dx y\right|\\
&+
\left|\int_{t_0}^{s_2-\eta}\int_BD^2_xG(x-y_1,t-s_2)-D^2_xG(x-y_2,t-s_2)\dx x\dx y\right.\\
&\qquad\left.\cdot\left[[f]_{\alpha(P),P,N}\left(|y_2-y_1|^{\alpha(P)}+|s_2-s_1|^{\alpha(P)/2}\right)+|f|_{0,N}\left(\frac{|y_2-y_1|}{d}+\frac{|s_2-s_1|}{d^2}\right)\right]\right|.
\end{align*}
Two first integrals on the right-hand side we estimate as was shown above. The last one we bound as follows
\begin{align*}
\int_{t_0}^{s_2-\eta}\int_B\int^{1}_{0}\left|D^3_xDG\left(x-y_1+z(y_1-y_2),t-s_2\right)\right||y_1-y_2|\dx z\dx x\dx t.
\end{align*}
We use inequality (\ref{helpineq2}) and we substitute $|x-y_1+z(y_1-y_2)|=\rho(s_2-t)^{1/2}$. It yields that the considered integral is bounded by
\begin{align*}
C\int_{t_0}^{s_2-\eta}\int^{1}_{0}(s_2-t)^{-3/2}\dx z\dx t|y_1-y_2|=\frac{C}{2}(y_1-y_2)\left(\eta^{-1/2}-(s_2-t_0)^{-1/2}\right)\leq C(y_1-y_2)\eta^{-1/2}\leq C.
\end{align*}
Finally, we have shown that
\begin{align}\label{kon4}
H'_1\leq C\left(|f|_{0,N}+d^{\alpha(P)}[f]_{\alpha(P),P,N}+d^{\alpha(Q)}[f]_{\alpha(Q),Q,N}\right).
\end{align}
Hence, we put (\ref{kon1}), (\ref{kon2}), (\ref{kon3}) and (\ref{kon4}) into (\ref{kon}) and then (\ref{konj}) and (\ref{kon}) into (\ref{lem1in1pom}). It yields
\begin{align*}
d^{\alpha(Q)}\frac{|D^2_yu(Q)-D^2_yu(P)|}{d^{\alpha(Q)}(Q,P)}\leq C\left(|f|_{0,N}+d^{\alpha(P)}[f]_{\alpha(P),P,N}+d^{\alpha(Q)}[f]_{\alpha(Q),Q,N}+|u|_{0,N}d^{-2}\right).
\end{align*}

\end{proof}

The next theorem will be analogous to the previous one, but we shall consider a general parabolic operator here.
\begin{theo}\label{lem2}
Let $\Omega\subset\R^n$ be an open and bounded set and $T>0$. Let  $N(P,d)\subset\ot$ and let $A=(a_{ij})_{ij}\in\R^{n\times n}$ be a matrix such that inequalities
$$
\lambda|\zeta|^2\leq \sum_{i,j=1}^na_{ij}\zeta_i\zeta_j\leq \Lambda|\zeta|^2\textrm{ for all }\zeta\in\R^n
$$
are satisfied for certain $\lambda>0$ and $\Lambda>0$.

Let $u\in \chp{2}{1}(\overline{N})$ be a solution of the equation  
$$
\sum_{i,j=1}^na_{ij}u_{x_ix_j}-u_t=f,
$$
where $f\in\chpbt(\overline{N})$. Then
\begin{align*}
|D^2u(Q)|&\leq C\left( |f|_{0,N}+d^{\alpha(Q)}[f]_{\alpha(Q),Q,N}+|u|_{0,N}d^{-2}\right),\\
d^{\alpha(Q)}\frac{|D^2_yu(Q)-D^2_yu(P)|}{d^{\alpha(Q)}(Q,P)}&\leq C\left(|f|_{0,N}+d^{\alpha(P)}[f]_{\alpha(P),P,N}+d^{\alpha(Q)}[f]_{\alpha(Q),Q,N}+|u|_{0,N}d^{-2}\right)
\end{align*}
hold for any $Q$ such that $d(P,Q)\leq \tfrac{d}{4}$ and $Q\in N$, where
  $C=C\left(\textnormal{diam}\left(\ot\right),n, \alpha^+,\alpha^-,c_{\log}(\alpha),\Lambda,\lambda \right)$.
\end{theo}
\begin{proof}
The proof is almost the same as the proof of Theorem \ref{lem1}. Instead of $G$ we use the following function
$$
\frac{\left(\textnormal{det}(a_{ij})\right)^{{1}/{2}}}{\left(2\sqrt{\pi}\right)^n}(s-t)^{-{n}/{2}}\expo{-\frac{\sum_{i,j=1}^na_{ij}(x_i-y_i)(x_j-y_j)}{4(s-t)}}.
$$
\end{proof}

In the next theorem we will consider a general parabolic equation on an arbitrary open set. We will prove Schauder interior estimates. For this purpose we need to introduce the interior norms.
Let $P=(x,t), Q\in\ot$. We define $d_P=\min\left\{t, \dist{\partial\Omega}{x}\right\}$, $d_{P,Q}=\min\{d_P,d_Q\}$. Now, we are able to introduce the following norms and seminorms
\begin{align*}
[u]_{k,0,\ot}^* &=[u]_{k,\ot}^*=\sup_{{P\in\ot}}d_P^k|D^k u(P)|, & [u]^{*}_{k,\alpha(\cdot),\Omega}&=\sup_{\substack{P,Q\in\ot\\P\neq Q}}d^{k+\alpha(P)}_{P,Q}\frac{|D^ku(P)-D^{k}u(Q)|}{d^{\alpha(P)}(P,Q)},\\
|u|^*_{k,\ot}&=\sum\limits_{j=0}^k[u]_{j,\ot}^*,&
|u|_{k,\alpha(\cdot),\ot}^*&=|u|_{k,\ot}^*+[u]_{k,\alpha(\cdot),\ot}^*.
\end{align*}
Next, for $s\in\R_+$ we define
\begin{align*}
[u]_{k,\ot}^{(s)}&=[u]_{k,0,\ot}^{(s)}=\sup_{\substack{P\in\ot }}d_P^{k+s}|D^{\beta}u(P)|,&
[u]_{k,\alpha(\cdot),\ot}^{(s)}&=\sup_{\substack{P,Q\in\ot\\ P\neq Q}}d^{k+\alpha(x)+s}_{x,y}\frac{|D^{k}u(x)-D^{k}u(y)|}{|x-y|^{\alpha(x)}},\\
|u|^{(s)}_{k,\ot}&=\sum\limits_{j=0}^k[u]^{(s)}_{j,\ot},&  |f|^{(s)}_{k,\alpha(\cdot),\Omega}&=|f|_{k,\Omega}^{(s)}+[f]^{(s)}_{k,\alpha(\cdot),\Omega}.
\end{align*}

\begin{theo}\label{sches}
Let $\Omega\subset \R^n$ be an open and bounded set and let $T>0$. If $u\in \chp{2}{1}(\otb)$ satisfies
\begin{equation}\label{equel}
u_t-Lu=u_t-\left(a^{ij}D_{ij}u+b^iD_iu+cu\right)\footnote{ We use there the Einstein summation convention.}=f,
\end{equation}
where $f\in C^{\alpha(\cdot)}(\otb)$ and there are positive constants $\lambda$ and $\Lambda$ such that
\begin{align*}
&a^{ij}(x,t)\zeta^i\zeta^j\geq\lambda|\zeta|^2\qquad\textrm{for all $x\in\Omega$ and $0\leq t<T$ and for all $\zeta\in\mathbb{R}^n$,}\\
&|a^{ij}|^{(0)}_{0,\alpha(\cdot),\ot},|b^i|_{0,\alpha(\cdot),\ot}^{(1)},|c|^{(2)}_{0,\alpha(\cdot),\ot}\leq \Lambda,
\end{align*} then
\begin{equation*}
|u|^*_{2,\alpha(\cdot), \ot}\leq C\left(|u|_{0, \ot}+|f|^{(2)}_{0,\alpha(\cdot), \ot}\right),
\end{equation*}
where $C=C(\textup{diam}(\Omega), n, \alpha^-, \alpha^+, c_{\log}(\alpha), \Lambda, \lambda)$. 
\end{theo}
\begin{proof}
From Lemma \ref{intlem} for $\epsilon>0$ we have
\begin{align*}
|u|_{2,0,\ot}^*\leq \epsilon[u]_{2,\ac,\ot}+C|u|_{0\,\ot},
\end{align*}
where $C=C(\epsilon)$. By this inequality it is sufficient to estimate the term~$[u]_{2,\ac,\ot}$.
There exist two points $P,Q\in \ot$ such that inequality 
\begin{align*}
\frac{1}{2}[u]_{2,\ac,\ot}^*\leq d^{\alpha(P)}_{P,Q}\frac{|D^2_xu(P)-D_x^2u(Q)|}{d^{\alpha(P)}(P,Q)}
\end{align*}
is satisfied. We can assume that $t$-coordinate of $P$ is greater than $t$-coordinate of $Q$.

Let us write equation (\ref{equel}) in the form
\begin{align}\label{eqmod}
a^{ij}(P)D_{ij}-u_t=\left(a^{ij}(P)-a^{ij}\right)D_{ij}u-b^iD_iu-cu+f=F.
\end{align}
Let us take arbitrary $\tfrac{1}{2}>\mu>0$.
We will consider two cases:
\begin{enumerate}
\item\label{przyp1} $d(P,Q)< \mu d_{P,Q}$

From this inequality we conclude $d(P,Q)<\mu d_{P}$. Let $d=\mu d_{P}$ and let us consider the cube $N(P,d)$. We apply Theorem \ref{lem1} to  equation (\ref{eqmod}) and then by inequality (\ref{tezin1}) we get
\begin{align}\label{in20031}
d^{\alpha(P)}\frac{|D^2_xu(P)-D_x^2u(Q)|}{d^{\alpha(P)}(P,Q)}\leq C\left(|F|_{0,N}+d^{\alpha(P)}[F]_{\alpha(P),P,N}+d^{\alpha(Q)}[F]_{\alpha(Q),Q,N}+|u|_{0,N}d^{-2}\right).
\end{align}

We shall estimate terms on the right-hand side of the above inequality. 
First, we will estimate the  term $|F|_{0,N}$. It is easy to see that  
\begin{align*}
d^{-2}|u|_{0,N}\leq \lambda^{-2}d_P^{-2}|u|_{0,\ot}.
\end{align*}

Next, we shall use the fact that for $Q\in N$ we have
\begin{align}\label{3in14}
d_Q\geq d_P-d=d_P-\lambda d_P >\frac{1}{2}d_P.
\end{align}

Thus, we bound the expression $|F|_{0,N}$ in the following way
\begin{align*}
|F|_{0,N}\leq |f|_{0,N}+\sum_{i=1}^n|b^i|_{0,N}|D_iu|_{0,N}+\sum_{i,j=1}^n|a^{ij}(P)-a^{ij}|_{0,N}|D_{ij}u|_{0,N}+|c|_{0,N}|u|_{0,N}.
\end{align*}
Now, we use inequality (\ref{3in14}) and we get
\begin{align}\label{1603in2}
|F|_{0,N}\leq \frac{C}{d_p^2}\left(|f|^{(2)}_{0,\ot}+|u|_{1,\ot}^*+\mu^{\alpha(P)}[u]_{2,\ot}^*\right).
\end{align}
Next, we will estimate the term $d^{\alpha(P)}[F]_{\alpha(P),P,N}$. For this purpose we need the following Leibniz-rule in H\"older space:
Let $g$ and $h$ be  arbitrary functions defined on $\ot$, then
\begin{align*}
[gh]_{\alpha(R),R,N}\leq |g(R)|[h]_{\alpha(R),R,N}+|h|_{0,N}[g]_{\alpha(R),R,N},
\end{align*}
where $R$ is a certain point from the set $\ot$. Thus, we have
\begin{align}\label{in30062}
d^{\alpha(P)}[F]_{\alpha(P),P,N}&\leq C d^{\alpha(P)}\left(d_{P}^{-2-\alpha(P)}[f]^{(2)}_{0,\ac, \ot} + \sum_{i=1}^n|b^i(P)|[D_xu]_{\alpha(P),P,N}+|c(P)|[u]_{\alpha(P),P,N}\right.\nonumber\\&\left.
+\sum_{i,j=1}^n[a^{ij}]_{\alpha(P),P,N}|D^2_xu|_{0,N}+\sum_{i=1}^n[b^i]_{\alpha(P),P,N}|D_xu|_{0,N}+[c]_{\alpha(P),P,N}|u|_{0,N}\right)\nonumber\\
\leq{}& C\left(\mu^{\alpha(P)}d_P^{-2}[f]^{(2)}_{0,\ac, \ot}+\mu^{\alpha(P)}d_P^{\alpha(P)}d_P^{-1}[D_xu]_{\alpha(P),P,N}+\mu^{\alpha(P)}d_P^{\alpha(P)}d_P^{-2}[u]_{\alpha(P),P,N}\right.\nonumber\\&\left.+\mu^{\alpha(P)}d^{-2}_P|u|_{2,0,\ot}^*\right)\leq C\mu^{\alpha(P)}d_P^{-2}\left([f]^{(2)}_{0,\ac, \ot}+|u|_{2,0,\ot}^*+[u]_{0,\ac,\ot}+[u]_{1,\ac,\ot}^*\right).
\end{align} 
Next, we will estimate the term $d^{\alpha(Q)}[F]_{\alpha(Q),Q,N}$. We do it in the similar way as in (\ref{in30062})
\begin{align}\label{in20032}
d^{\alpha(Q)}[F]_{\alpha(Q),Q,N}\leq{} & C\lambda^{\alpha(Q)}d_P^{-2}\left([f]^{(2)}_{0,\ac, \ot}+|u|_{2,0,\ot}^*+[u]_{0,\ac,\ot}+[u]_{1,\ac,\ot}^*\right)\\&
+Cd^{\alpha(Q)}\sum_{i,j=1}^n|a^{ij}(P)-a^{ij}(Q)|[D^2_xu]_{\alpha(Q),Q,N}.
\nonumber
\end{align}
We shall estimate the last term on the right-hand side of the above inequality. We use there inequality~(\ref{3in14})
\begin{align*}
d^{\alpha(Q)}|a^{ij}(P)-a^{ij}(Q)|[D^2_xu]_{\alpha(Q),Q,N}\leq{}& C d^{\alpha(Q)}d^{\alpha(Q)}(P,Q)d^{-\alpha(Q)}_{P,Q}d_P^{-\alpha(Q)-2}|a^{ij}|^*_{0,\ac,\ot}[u]^*_{2,\ac,\ot}\\
={}&C\mu^{2\alpha(Q)}d_P^{-2}[u]^*_{2,\ac,\ot}.
\end{align*}
We put the above inequality into (\ref{in20032}), what yields
\begin{align}\label{in30061}
d^{\alpha(Q)}[F]_{\alpha(Q),Q,N}
&\leq C\mu^{\alpha(Q)}d_P^{-2}\left([f]^{(2)}_{0,\ac, \ot}+|u|_{2,0,\ot}^*+[u]_{0,\ac,\ot}+[u]_{1,\ac,\ot}^*\right)\nonumber\\
&+C\mu^{2\alpha(Q)}d_P^{-2}[u]^*_{2,\ac,\ot}.
\end{align}
Then, we put inequalities (\ref{1603in2}), (\ref{in30062}) and (\ref{in30061}) into (\ref{in20031}) and  we get
\begin{multline}\label{2203in1}
d^{\alpha(P)}\frac{|D^2_xu(P)-D_x^2u(Q)|}{d^{\alpha(P)}(P,Q)} \leq 
\mu^{\alpha(Q)}d_P^{-2}C\left(|f|^{(2)}_{0,\ac, \ot}+|u|_{2,0,\ot}^*\right)+C\mu^{2\alpha(Q)}d_P^{-2}[u]^*_{2,\ac,\ot}\\+C d^{-2}|u|_{0,\ot}+\frac{C}{d_p^2}\left(|f|^{(2)}_{0,\ot}+|u|_{1,\ot}^*+\mu^{\alpha(P)}[u]_{2,\ot}^*\right)+C\mu^{\alpha(P)}d_P^{-2}\left([f]^{(2)}_{0,\ac, \ot}+|u|_{2,0,\ot}^*\right),
\end{multline}
where we have used the interpolation inequality from Lemma \ref{intlem}.

We need only to control the number $\mu^{\alpha(Q)-\alpha(P)}$ to finish the proof. Here we have two cases.
\begin{itemize}
\item[(a)] If $\alpha(Q)-\alpha(P)\geq 0$, then we proceed in the following way
\begin{align*}
\mu^{\alpha(Q)-\alpha(P)}\leq \left(\frac{1}{2}\right)^{\alpha(Q)-\alpha(P)}\leq \left(\frac{1}{2}\right)^{\alpha^--\alpha^+}.
\end{align*}
\item[(b)] If $\alpha(Q)-\alpha(P) <0$, then 
\begin{align*}
\mu^{\alpha(Q)-\alpha(P)}&\leq \left(\frac{d(P,Q)}{d_P}\right)^{\alpha(Q)-\alpha(P)}\leq e^{C_{\log}(\alpha)}d_P^{\alpha(P)-\alpha(Q)}\leq e^{C_{\log}(\alpha)}\max\left(1,{\rm diam}(\ot)\right)^{\alpha(P)-\alpha(Q)}\\
&\leq  e^{C_{\log}(\alpha)}\max\left(1,{\rm diam}(\ot)\right)^{\alpha^+-\alpha^-}.
\end{align*}
\end{itemize}
Finally, from inequality (\ref{2203in1}) we have 
\begin{align}\label{2203in2}
d^{\alpha(P)}\frac{|D^2_xu(P)-D_x^2u(Q)|}{d^{\alpha(P)}(P,Q)}\leq Cd_P^{-2}\left(|f|^{(2)}_{0,\ac, \ot}+|u|_{0,\ot}+\mu^{2\alpha(P)}|u|_{2,\ac,\ot}\right).
\end{align}
\item\label{przyp2} $d(P,Q)\geq \mu d_{P,Q}$ 

In this case we proceed in the following way
\begin{align*}
d^{\alpha(P)}_{P,Q}\frac{|D^2_xu(P)-D_x^2u(Q)|}{d^{\alpha(P)}(P,Q)}\leq C\mu^{-\alpha(P)}d_P^{-2}[u]^*_{2,0,\ot}\leq d_P^{-2} C\left(|u|_{0,\ot}+\mu^{2\alpha(P)}[u]^*_{2,\ac,\ot}\right).
\end{align*}
\end{enumerate}
We see that when we join together the inequalities from cases \ref{przyp1} and \ref{przyp2}, then we obtain the same inequality as in (\ref{2203in2}).

Now, we will estimate from below the left-hand side of  inequality (\ref{2203in2})
\begin{align*}
d^{\alpha(P)}\frac{|D^2_xu(P)-D_x^2u(Q)|}{d^{\alpha(P)}(P,Q)}&=\mu^{\alpha(P)}d_P^{\alpha(P)}\frac{|D^2_xu(P)-D_x^2u(Q)|}{d^{\alpha(P)}(P,Q)}\\
&\geq \mu^{\alpha(P)} d_P^{-2}d_{P,Q}^{2+\alpha(P)}\frac{|D^2_xu(P)-D_x^2u(Q)|}{d^{\alpha(P)}(P,Q)}\geq \mu^{\alpha(P)} d_P^{-2}\frac{1}{2}|u|_{2,\ac,\ot}.
\end{align*}
When we put the above inequality into (\ref{2203in2}), we get 
\begin{align*}
|u|_{2,\ac,\ot}&\leq \mu^{-\alpha(P)}C\left( |f|^{(2)}_{0,\ac, \ot}+|u|_{0,\ot}\right)+ C\mu^{\alpha(P)}|u|_{2,\ac,\ot} \\
&\leq
\mu^{-\alpha^+}C\left( |f|^{(2)}_{0,\ac, \ot}+|u|_{0,\ot}\right)+ C\mu^{\alpha^-}|u|_{2,\ac,\ot}.
\end{align*}
We take  $\mu\leq 1/2$, such that $C\mu^{\alpha^-}=1/2$. Then, we subtract $C\mu^{\alpha^-}|u|_{2,\ac,\ot}=\frac{1}{2} |u|_{2,\ac,\ot}$ from both sides of the previous inequality and we finish the proof.
\end{proof}

\section{Boundary estimates}\label{bousec}
This part of the paper is devoted to proving the boundary Schauder estimates.
Let $y_0=(y_{0,1},\ldots,y_{0,n})\in \mathbb{R}^n$ and let us take $d>0$ and $\bar{d}\leq d$. We define
\begin{align*}
B=\left\{y\colon |y_i-y_{0,i}|\leq d \textrm{ for } i=1,\ldots,n-1 \textrm{ and } -d\leq y_n-y_{0,n}\leq \bar{d}\right\}
\end{align*}
and $N_s=B\times [0,s]$. Let us take $s_0$ such that $s_0\leq d^2$. We assume $s\leq s_0$ we denote $N=N_{s_0}$, $P=(y_0,s_0)$.

\begin{theo}\label{schesbrz}
Let $\Omega\subset\R^n$ be an open and bounded set and $T>0$. Let us assume that $N\subset\otb$ and let $f\in \chpbt (\overline{N})$, $u\in \chp{2}{1}(\overline{N})$ satisfy equation
\begin{align}\label{equ1}
\Delta u-u_t=f
\end{align}
and let $u=0$ on sets $\left\{y_n=\bar{d}\right\}$ and $\left\{t=0\right\}$, then the 
following inequalities
\begin{align}
|D^2u(P)|\leq {}&C\left( |f|_{0,N}+d^{\alpha(P)}[f]_{\alpha(P),P,N}+|u|_{0,N}d^{-2}\right),\label{teez1}\\\begin{split}
d^{\alpha(Q)}\frac{|D^2u(Q)-D^2u(P)|}{d^{\alpha(Q)}(Q,P)}\leq {} &
C\left(|f|_{0,N}+d^{\alpha(P)}[f]_{\alpha(P),P,N}+d^{\alpha(y_0,s)}[f]_{\alpha(y_0,s),(y_0,s),N}\right.\\ & \qquad\left.+d^{\alpha(y,s_0)}[f]_{\alpha(y,s_0),(y,s_0),N}+d^{\alpha(Q)}[f]_{\alpha(Q),Q,N}+|u|_{0,N}d^{-2}\right)
\end{split}\label{teez2}
\end{align}
are satisfied for $Q=(y,s)\in N$ and $d(P,Q)\leq \frac{d}{4}$.
\end{theo}
\begin{proof}

We will use the same notation as in the proof of Theorem \ref{lem1}. Moreover, let
\begin{align*}
\overline{G}=G(x,t;y,s)-G(x,t;y^*,s),\textrm{ where $y^*=(y_1,\ldots,y_{n-1},2\bar{d}-y_n)$ for $y=(y_1,\ldots, y_n)$.
}
\end{align*}

\textbf{1.} Integrating over $B\times [0,s]$ the Green identity (\ref{green}) with $u$ and $v=\varphi \bar{G}$ we get
\begin{align}\label{1905in1}
\begin{split}
u(y,s)=&-\int_{0}^s\int_B \phi(x,t) \overline{G}(x-y;t-s) f(x,t)\dx x\dx t+\int_{0}^s\int_Bu(x,t)L_0^*\left(\phi(x,t) \overline{G}(x-y;t-s)\right)\dx x\dx t\\
=&-H_0+J_0,
\end{split}
\end{align}
where $Q=(y,s)\in N$.

Since $|y^*-x|\geq |y-x|$ by inequality (\ref{helpineq2}) we have
\begin{align}\label{brzhelp1}
\left|D^i_tD^j_xD^k_sD^l_y \overline{G}(x,t;y,s)\right|\leq C (s-t)^{-(2i+j+2k+l+n)/2}\expo{-\frac{|x-y|^2}{5(s-t)}}.
 \end{align}
 Let us observe that for $i\neq n$
\begin{align}\label{prop30061}
 \frac{\partial^2}{\partial y_j\partial y_i}\int_0^s\int_B\overline{G}(x,t;y,s)\dx x\dx t<\infty.
\end{align}
Indeed,
\begin{align*}
\frac{\partial^2}{\partial y_j\partial y_i}\int_0^s\int_B\overline{G}(x,t;y,s)&\dx x\dx t=\frac{\partial}{\partial y_j}\int_0^s\int_B\overline{G}_{y_i}(x,t;y,s)\dx x\dx t = \frac{\partial}{\partial y_j}\int_0^s\int_B\overline{G}_{x_i}(x,t;y,s)\dx x\dx t \\&
=\frac{\partial}{\partial y_j}\int_0^s\int_{\partial B}\overline{G}(x,t;y,s)n_i\dx S\dx t=\int_0^s\int_{\partial B}\overline{G}_{y_j}(x,t;y,s)n_i\dx S\dx t\\
&=\int_0^s\int_{\partial B\cap \{x_i=d+y_{0,i}\}}\overline{G}_{y_j}(x,t;y,s)\dx S\dx t-\int_0^s\int_{\partial B\cap \{x_i=-d+y_{0,i}\}}\overline{G}_{y_j}(x,t;y,s)\dx S\dx t
\\&\hskip-3.5pt\stackrel{(\ref{brzhelp1})}{\leq} C\int_0^s\int_{\partial B} (s-t)^{-(1+n)/2}\expo{-\frac{|x-y|^2}{5(s-t)}}\dx S\dx t.
\end{align*}
Using the assumption $d(P,Q)\leq \frac{d}{4}$ for $x\in\partial B$, we have $|x-y|\geq \frac{3d}{4}$, so  the last integral in the above inequality  can be  estimated by the following expression
\begin{align*}
\int_0^s\int_{\partial B} (s-t)^{-(1+n)/2}\expo{-C\frac{d^2}{(s-t)}}\dx S\dx t=Cd^{n-1}\int_0^s(s-t)^{-(1+n)/2}\expo{-C\frac{d^2}{(s-t)}}\dx t.
\end{align*}
We substitute $z=\frac{d^2}{s-t}$ into the  above integral and we get
\begin{align*}
Cd^{n-1}\int^{\infty}_{d^2/s} z^{(n+1)/2}d^{-1-n}\expo{-Cz}\frac{z^2}{d^2}\dx z\leq C\int^{\infty}_{0}z^{(n+5)/2}\expo{-Cz}\dx z\leq C 
\end{align*}
and (\ref{prop30061}) follows.

We can obtain (\ref{teez1}) in the similar way as in the proof of Theorem \ref{lem1}. The derivatives 
\begin{align*}
\frac{\partial^2 u}{\partial y_i\partial y_j}\textrm{ for }i\neq n\textrm{ or } j\neq n
\end{align*}
we can estimate almost step by step as in the proof of Theorem \ref{lem1}. It is left to estimate the derivative 
$
\frac{\partial^2u}{\partial^2y_n}$.

\textbf{2.} Because $u$ satisfies equation (\ref{equ1}), it suffices to prove inequality (\ref{teez2}) for $u_t$.
 We differentiate equality (\ref{1905in1}) with respect to $s$ and we get
\begin{align}\label{3005eq}
\frac{\partial u(y,s)}{\partial s}=-H(Q)+J(Q),
\end{align}
where
\begin{align*}
H(Q)=\int_0^s\int_B\phi(x,t)\overline{G}_s(x-y;t-s)f(x,t)\dx x\dx t+f(y,s)\phi(y,s)
\end{align*}
and
\begin{align}\label{row471}
\begin{split}
J(Q)&=\int_{0}^s\int_Bu(x,t)L_0^*\left(\phi(x,t) \overline{G}_s(x-y;t-s)\right)\dx x\dx t+\lim_{t\to s^-}\int_Bu(x,t)L_0^*\left(\phi(x,t) \overline{G}(x-y;t-s)\right)\dx x\\
&=J_1(Q)+J_2(Q).
\end{split}
\end{align}
First, we will bound $J$
\begin{align}\label{in2305}
\begin{split}
J_2(Q)=&\\
\lim_{t\to s^-}\int_B& u(x,t)\left(\phi_t(x,t)\overline{G}(x-y;t-s)+\Delta\phi(x,t)\overline{G}(x-y;t-s)+2D\phi(x,t)\cdot D_x \overline{G}(x-y;t-s) \right)\dx x.
\end{split}
\end{align}
Now, we shall compute the third term in the above equality
\begin{align*}
\int_Bu(x,t) D\phi(x,t)\cdot D_x \overline{G}(x-y;t-s)\dx x={}&\int_{\partial B}u(x,t) D\phi(x,t)\cdot n(x)  \overline{G}(x-y;t-s)\dx S(x)\\-\int_B {\rm div}_x\left(u(x,t)D\phi(x,t)\right) \overline{G}(x-y;t-s)\dx x={}&
-\int_B {\rm div}_x\left(u(x,t)D\phi(x,t)\right) \overline{G}(x-y;t-s)\dx x.
\end{align*}
We plug the above  equality into (\ref{in2305}) and we get
\begin{align*}
J_2(Q)&=\lim_{t\to s^-}\int_B\left(u(x,t)\phi_t(x,t)+u(x,t)\Delta\phi(x,t)-{\rm div}_x\left(u(x,t)D\phi(x,t)\right)\right)\overline{G}(x-y;t-s)\dx x\\
&=u(y,s)\phi_t(y,s)+u(y,s)\Delta\phi(y,s)-{\rm div}_x\left(u(y,s)D\phi(y,s)\right),
\end{align*}
what is equal to $0$ because $d(P,Q)\leq \frac{d}{4}$ and $\phi(Q)=1$.
Thus, we have
\begin{align}\label{equ471}
\begin{split}
&J(Q)=J_1(Q)=\\
&\int_0^s\int_B u(x,t)\big(\phi_t(x,t)\overline{G}(x-y;t-s)+\Delta\phi(x,t)\overline{G}(x-y;t-s)+2D\phi(x,t)\cdot D_x \overline{G}(x-y;t-s) \big)\dx x.
\end{split}
\end{align}
This term we bound as term (\ref{2505in}) in the proof of Theorem \ref{lem1}.
The integral in $H(Q)$ we bound as $H$ in inequality (\ref{kini}). We use there inequality (\ref{brzhelp1}). 

\textbf{3.} Now, we shall estimate the H\"older semi-norm of $D^2u$.
Derivatives 
\begin{align*}
\frac{\partial^2 u}{\partial y_i\partial y_j}\ {\rm for }\ i\neq n\ {\rm or }\ j\neq n
\end{align*}
 we estimate in the similar way as in the proof of Theorem \ref{lem1}.
The derivative $\frac{\partial^2 u}{\partial^2 y_n}$  we shall control by (\ref{equ1}) and by the estimation for $u_t$.
Hence, now we need to estimate the H\"older semi-norm of $u_t$. By equality (\ref{3005eq}) we have
\begin{align*}
d^{\alpha(P)}\frac{|u_s(P)-u_s(Q)|}{d^{\alpha(P)}(P,Q)}\leq d^{\alpha(P)}\frac{|H(P)-H(Q)|}{d^{\alpha(P)}(P,Q)}+d^{\alpha(P)}\frac{|J(P)-J(Q)|}{d^{\alpha(P)}(P,Q)}.
\end{align*}
As in equality (\ref{equ471})  we can show that 
\begin{align*}
J(P)=J_1(P)\qquad\textrm{and} \qquad J(Q)=J_1(Q), 
\end{align*}
where $J_1$ is given in  (\ref{row471}). Thus, we get
\begin{align*}
d^{\alpha(P)}\frac{|J(P)-J(Q)|}{d^{\alpha(P)}(P,Q)}=d^{\alpha(P)}\frac{|J_1(P)-J_1(Q)|}{d^{\alpha(P)}(P,Q)}.
\end{align*}
We estimate the above term in similar way as the term $J'$ in inequality (\ref{lem1in1pom}).

It is left to bound 
\begin{align*}
d^{\alpha(P)}\frac{|H(P)-H(Q)|}{d^{\alpha(P)}(P,Q)}.
\end{align*}
For this purpose we apply Lemma \ref{lemosz}. In this Lemma we can change $G$ into $\overline{G}$. Let us rewrite $H(P)-H(Q)$ to use this result. We define $g(Q)=f(Q)\phi(Q)$ and then we have
\begin{align}\label{eq5021}
H(P)-H(Q)=\frac{\partial}{\partial s}\int_0^{s_0}\int_Bg(x,t)\overline{G}(x,t;y_0,s_0)\dx x\dx t- \frac{\partial}{\partial s}\int_0^s\int_Bg(x,t)\overline{G}(x,t;y,s)\dx x\dx t.
\end{align}
We can assume $y_0=0$.
We transform $(x,t)\to(\bar{x},\bar{t})\colon N\to U$ (we use the notion from Lemma \ref{lemosz}) in the following way $\bar{x}=x/d$ and $\bar{t}=t/d$. So we transform $P$ into $\bar{P}$ and $Q$ into $\bar{Q}$. Let us also see that $B$ is transformed into $I=(-1,1)\times\times\ldots\times (-1,1)\times (-1,\beta)\subset\R^n$, where $\beta=\frac{\bar{d}}{{d}}$. Thus, the first integral in (\ref{eq5021}) we can rewrite in the following way
\begin{align*}
\int_0^{\bar{s}_0}\int_I\bar{g}(\bar{x},\bar{t})\overline{G}(\bar{x},\bar{t};\bar{y}_0,\bar{s}_0)\dx \bar{x}\dx \bar{t},
\end{align*}
where $\bar{g}(\bar{x},\bar{t})=g(d\bar{x},d^2\bar{t})$. The second integral from (\ref{eq5021}) we can rewrite in the analogous way.

All of the assumptions of Lemma \ref{lemosz} are satisfied  in the easy way. 
Thus, thanks to this Lemma the proof is finished.
\end{proof}
\begin{theo}\label{schesbrzog}
Let us assume that sets $N$ and $\otb$ satisfy the same conditions as in Theorem \ref{schesbrz} and let $A=(a_{ij})_{ij}\in\R^{n\times n}$ be a matrix such that inequalities
$$
\lambda|\zeta|^2\leq \sum_{i,j=1}^na_{ij}\zeta_i\zeta_j\leq \Lambda|\zeta|^2\textrm{ for all }\zeta\in\R^n
$$
are satisfied for certain $\lambda>0$ and $\Lambda>0$.

Let $u\in \chp{2}{1}(\overline{N})$ satisfies the equation  
$$
Lu=\sum_{i,j=1}^na_{ij}u_{x_ix_j}-u_t=f,
$$
where $f\in\chpbt(\overline{N})$. Then 
\begin{align*}
|D^2u(P)|\leq {}&C\left( |f|_{0,N}+d^{\alpha(P)}[f]_{\alpha(P),P,N}+|u|_{0,N}d^{-2}\right),\\
d^{\alpha(Q)}\frac{|D^2u(Q)-D^2u(P)|}{d^{\alpha(Q)}(Q,P)}\leq{} &
C\left(|f|_{0,N}+d^{\alpha(P)}[f]_{\alpha(P),P,N}+d^{\alpha(y_0,s)}[f]_{\alpha(y_0,s),(y_0,s),N}\right.\\ & \qquad\left.+d^{\alpha(y,s_0)}[f]_{\alpha(y,s_0),(y,s_0),N}+d^{\alpha(Q)}[f]_{\alpha(Q),Q,N}+|u|_{0,N}d^{-2}\right)
\end{align*}
hold for any $Q=(y,s)\in N$ such that $d(P,Q)\leq \frac{d}{4}$, where
  $C=C\left(\textnormal{diam}\left(\ot\right),n, \alpha^+,\alpha^-,c_{\log}(\alpha),\Lambda,\lambda \right)$.
\end{theo}
\begin{proof}
Let us assume that $y_0$, the center of $N$ is equal $0$.
There exists a matrix $P\in\R^{n\times n}$ such that matrix $\widetilde{A}=PAP^T$ is the unit matrix and $PB=[-\beta_1,\beta_1]\times\ldots\times [-\beta_n,\beta_n\theta]$, where $\beta_1,\ldots,\beta_n>0$ and $0<\theta\leq 1$ (for details see the proof of Lemma 4.1 in \cite{bies1}). Let us introduce the following notations $\widetilde{B}=PB$ and $\widetilde{N}=\widetilde{B}\times [0,s_0]$.
We define 
\begin{align*}
\tilde{u}(y,t)=u(P^{-1}y,t),\qquad\tilde{f}(y,t)=f(P^{-1}y,t)
\end{align*}
for $(y,t)\in \widetilde{N}$.
It is easy to check that $\tilde{u}$ and $\tilde{f}$ satisfy the equation
\begin{align}\label{eq190301}
\Delta \tilde{u} -\tilde{u}_t=\tilde{f}.
\end{align}

Theorem \ref{schesbrz} could be also proved for sets of type $\widetilde{N}$. Thus, we can use this Lemma for equation (\ref{eq190301}). Then, we transform $\widetilde{N}$ into $N$, $\tilde{u}$ into $u$ and $\tilde{f}$ into $f$ and in this way we finish the proof of the Lemma.
\end{proof}

The next theorem deals with boundary Schauder estimates on general open set. We need boundary norms and seminorms to prove it.
We define a parabolic boundary of $\ot$ as $\gt=\partial\Omega_T\setminus \Omega\times\{T\}$.
Let $\Gamma\subset \gt$.  For $P, Q\in\ot$ we define $\bar{d}_P=\dist{P}{\gt\setminus\Gamma}$, $\bar{d}_{P,Q}=~\min(\bar{d}_P, \bar{d}_Q)$. 
In the sequel, we shall use the following notations
\begin{align*}
[u]^*_{k,0,\otg}=[u]^*_{k,\otg} ={}& \sup_{\substack{P\in\ot
 }}\bar{d}_P^k\left|D^k(P)\right|,&	[u]^*_{k,\alpha(\cdot),\otg}={}&\sup_{\substack{P,Q\in\ot \\P\neq Q}} \bar{d}_{P,Q}^{k+\alpha(P)}\frac{\left|D^ku(P)-D^ku(Q)\right|}{d^{\alpha(P)}(P,Q)},	\\
|u|^*_{k,0,\otg}={}&|u|^*_{k,\otg}=\sum\limits_{j=0}^k[u]_{j,\otg}^*,&		|u|^*_{k,\alpha(\cdot),\otg}={}&|u|^*_{k,\otg}+[u]_{k,\alpha(\cdot),\otg}^*.
\end{align*}
In the analogous way we define norms $|\cdot|^{(s)}_{k,\ac,\otg}$ and respect seminorms.

We get the following theorem by an application of Theorem \ref{schesbrz}. 
\begin{theo}\label{thschesbrz}
Let $\Omega\subset \R^n$ be an open and bounded set and let $T>0$. Moreover, let $\Gamma$ be a  portion of $\gt$ contained in  $\{x_n=0\}\cup\{t=0\}$. 
If $u\in \chp{2}{1}(\otb)$ satisfies
\begin{align*}
u_t-Lu=u_t-\left(a^{ij}D_{ij}u+b^iD_iu+cu\right)=f,\quad u|_{\Gamma}=0,
\end{align*}
where $f\in C^{\alpha(\cdot)}(\otb)$ and there are positive constants $\lambda$ and $\Lambda$ such that
\begin{align*}
&a^{ij}(x,t)\zeta^i\zeta^j\geq\lambda|\zeta|^2\qquad\textrm{for all $x\in\Omega$ and $0\leq t<T$ and for all $\zeta\in\mathbb{R}^n$,}\\
&|a^{ij}|^{(0)}_{0,\alpha(\cdot),\otg},|b^i|_{0,\alpha(\cdot),\otg}^{(1)},|c|^{(2)}_{0,\alpha(\cdot),\otg}\leq \Lambda,
\end{align*} then
\begin{equation*}
|u|^*_{2,\alpha(\cdot), \otg}\leq C\left(|u|_{0, \ot}+|f|^{(2)}_{0,\alpha(\cdot), \otg}\right),
\end{equation*}
where $C=C(\textup{diam}(\Omega), T, n, \alpha^-, \alpha^+, c_{\log}(\alpha), \Lambda, \lambda,\Gamma)$.
\end{theo}
\begin{proof}
The proof is analogous to the proof of Theorem \ref{sches}. There we use Theorem \ref{schesbrzog} instead of Theorem \ref{lem2}.
\end{proof}
\section{Global estimates}\label{glosec}
We shall prove global Schauder estimates in this section. We need the following Lemma.
\begin{lem}\label{brzglscheslem}
Let us assume that $\Omega$ is a set of class $C^{2,\alpha^+}$. 
If $u\in\chp{2}{1}(\otb)$ satisfies
\begin{align*}
u_t-Lu={}&f\textrm{ on }\ot,\\
u={}&\, 0\textrm{ on }\gt=\partial\Omega_T\setminus \Omega\times\{T\},
\end{align*}
where $f\in C^{\alpha(\cdot)}(\otb)$ and there are positive constants $\lambda$ and $\Lambda$ such that
\begin{align*}
&a^{ij}(x,t)\zeta^i\zeta^j\geq\lambda|\zeta|^2\qquad\textrm{for all $x\in\Omega$ and $0\leq t<T$ and for all $\zeta\in\mathbb{R}^n$,}\\
&|a^{ij}|_{0,\alpha(\cdot),\ot},|b^i|_{0,\alpha(\cdot),\ot},|c|_{0,\alpha(\cdot),\ot}\leq \Lambda,
\end{align*} then there exists $\rho>0$ such that for all $P\in \partial\Omega\times [0,T]\cup\Omega\times\{0\}$ the following inequality
\begin{equation*}
|u|_{2,\alpha(\cdot), \Omega_T\cap B(P,\delta)}\leq C\left(|u|_{0, \ot}+|f|_{0,\alpha(\cdot), \ot}\right)
\end{equation*}
holds and $C=C(\textup{diam}(\Omega), T, n, \alpha^-, \alpha^+, c_{\log}(\alpha), \Lambda, \lambda)$.
\end{lem}
\begin{proof}
Let us take $P=(x_0,t_0)\in\gt$. Because the boundary of $\Omega$ is of class $C^{2,\alpha^+}$, so we 
 have $\delta >0$ and injective mapping $\Phi\colon B=B(P,\delta)\to D \subset\R^n\times\R_+$ of class $C^{2,\alpha^+}$ such that the following conditions are valid 
\begin{align*}
\Phi\left(B\cap \ot\right)\subset D\cap\left(\R_+^n\times\R_+\right),\qquad
\Phi(B\cap \gt)\subset D\cap\left(\{x_n=0\}\times\R_+\cup\R^n_+\times\{0\}\right).
\end{align*}

Now, we transform our equation. It is similar path as in the proof of the analogous Lemma for elliptic equations (see the proof of Lemma 4.4 in \cite{bies1}).
Let us denote $B'=B\cap \ot$, $D'=D\cap \left(\R_+^n\times\R_+\right)$, $\Gamma=B\cap\gt$ and $\Gamma'=\Phi(\Gamma)$ . Next, we define $\tilde{u}=u\circ\Phi^{-1}$ and $\tilde{f}=f\circ\Phi^{-1}$ on the set $D'$. Then $\tilde{u}$ satisfies the equation
\begin{align*}
\tilde{u}_t-\widetilde{L}\tilde{u}=\tilde{u}_t-\left(\tilde{a}^{ij}D_{ij}\tilde{u}+\tilde{b}^iD_i\tilde{u}+\tilde{c}\tilde{u}\right)=\tilde{f},
\end{align*}
where
\begin{eqnarray*}
\tilde{a}^{ij}=\left(\sum_{k,l=1}^na^{kl}D_l\Phi^jD_k\Phi^i\right)\circ\Phi^{-1}, \quad 
\tilde{b}^i=\left(\sum_{k,l=1}^nD_{lk}\Phi^ia^{lk}+\sum_{k=1}^nb^kD_k\Phi^i\right)\circ\Phi^{-1},\quad
\tilde{c}=c\circ\Phi^{-1}.
\end{eqnarray*}
It is easy to see that there exists constant $K>0$ such that
\begin{align*}
K^{-1}d(X,Y)\leq d(\Phi(X),\Phi(Y))\leq Kd(X,Y) \textrm{ for } X,Y\in B(P,\delta).
\end{align*}
Thus, we obtain
\begin{equation}\label{pominbrz1}
\begin{split}
C_2|v|^*_{k,{\alpha}(\cdot),B'}\leq|\tilde{v}|^*_{k,\tilde{\alpha}(\cdot),D'}\leq C_1|v|^*_{k,{\alpha}(\cdot),B'}\\
C_2|v|^{(l)}_{k,\alpha(\cdot),B'}\leq|\tilde{v}|^{(l)}_{k,\tilde{\alpha}(\cdot),D'}\leq C_1|v|^{(l)}_{k,\alpha(\cdot),B'},
\end{split}
\end{equation}
and
\begin{equation}\label{pominbrz}
\begin{split}
C_2|v|^*_{k,\alpha(\cdot),B'\cup \Gamma}\leq|\tilde{v}|^*_{k,\tilde{\alpha}(\cdot),D'\cup \Gamma'}\leq C_1|v|^*_{k,\alpha(\cdot),B'\cup \Gamma}\\
C_2|v|^{(l)}_{k,\alpha(\cdot),B'\cup \Gamma}\leq|\tilde{v}|^{(l)}_{k,\tilde{\alpha}(\cdot),D'\cup{\Gamma}'}\leq C_1|v|^{(l)}_{k,\alpha(\cdot),B'\cup \Gamma},
\end{split}
\end{equation}
for $k, l=0,1, 2, 3,\ldots $, where $v$ is a certain function, $\tilde{v}=v\circ\Phi^{-1}$ and $\tilde{\alpha}=\alpha\circ\Phi^{-1}$. 

Hence, we see that 
\begin{align*}
|\tilde{a}^{ij}|_{0,\tilde{\alpha}(\cdot),D'},|\tilde{b}^i|_{0,\tilde{\alpha}(\cdot),D'},|\tilde{c}|_{0,\tilde{\alpha}(\cdot),D'}\leq \widetilde{\Lambda}=C\Lambda.
\end{align*}
Thus, by virtue of Theorem \ref{thschesbrz} we get
\begin{equation*}
|\tilde{u}|^*_{2,\tilde{\alpha}(\cdot), D'\cup \Gamma'}\leq C\left(|\tilde{u}|_{0,D'}+|\tilde{f}|_{0,\tilde{\alpha}(\cdot),D'\cup \Gamma'}^{(2)}\right).
\end{equation*}
Therefore, from (\ref{pominbrz1}) and (\ref{pominbrz}) we obtain
\begin{align}\label{schesexin1}
\begin{split}
|u|^*_{2,\alpha(\cdot),B'\cup \Gamma}&\leq |\tilde{u}|^*_{2,\tilde{\alpha}(\cdot), D'\cup \Gamma'}\leq C\left(|\tilde{u}|_{0,D'}+|\tilde{f}|_{0,\tilde{\alpha}(\cdot),D'\cup\Gamma'}^{(2)}\right)\leq \\
 &C\left(|u|_{0,B'}+|f|^{(2)}_{0,\alpha(\cdot),B' \cup \Gamma'}\right)\leq  C\left(|u|_{0,B'}+|f|_{0,\alpha(\cdot),B'}\right) \leq C\left(|u|_{0,\Omega}+|f|_{0,\alpha(\cdot),\Omega}\right).
 \end{split}
\end{align}
We denote $B''=B\left(P,\frac{\delta}{2}\right)\cap\Omega$. We see that $\bar{d}_X, \bar{d}_{X,Y}\geq \frac{\delta}{2}$ for all $X, Y\in B''$ and thus we conclude
\begin{align*}
C(\delta)|u|_{2,\alpha(\cdot),B''}\leq|u|^*_{2,\alpha(\cdot), B'\cup \Gamma}.
\end{align*}
According to inequality (\ref{schesexin1}), we have
\begin{align}\label{schesexin2}
|u|_{2,\alpha(\cdot),B''} \leq C(\delta)\left(|u|_{0,\ot}+|f|_{0,\alpha(\cdot),\ot}\right).
\end{align}

We denote $\rho_P=\delta/4$.
Now, let us take the covering  $\{B(P,{\rho_P})\}_{P\in\mathcal{G}}$ of the set $\gt$. 
Since this set is compact, we can take a finite cover  $\{B(P_i,{\rho_{P_i}}\}_{i=1}^{N}$ of $\gt$. Let $\rho=\min{\rho_{P_i}}$. 
If we take an arbitrary $X\in\gt$, then $X\in B(P_i, {\rho_{P_i}})$ for some $i$. It is easy to see that $B(X,\rho)\subset B(P_i, 2{\rho_{P_i}})$ 
and since for $B''=B(P_i,2{\rho_i})\cap\ot$ inequality (\ref{schesexin2}) holds. Thus, the proof follows.
\end{proof}

The next theorem is the main result in this section.
\begin{theo}
Let us assume that $\Omega$ is a set of class $C^{2,\alpha^+}$. 
If $u\in \chp{2}{1}(\otb)$ satisfies
\begin{align*}
u_t-Lu={}&f\textrm{ on }\ot,\\
u={}& \varphi\textrm{ on }\gt,
\end{align*}
where $f\in C^{\alpha(\cdot)}(\otb)$, $\varphi\in C^{2,\ac}(\otb)$ and there are positive constants $\lambda$ and $\Lambda$ such that
\begin{align*}
&a^{ij}(x,t)\zeta^i\zeta^j\geq\lambda|\zeta|^2\qquad\textrm{for all $x\in\Omega$ and $0\leq t<T$ and for all $\zeta\in\mathbb{R}^n$,}\\
&|a^{ij}|_{0,\alpha(\cdot),\ot},|b^i|_{0,\alpha(\cdot),\ot},|c|_{0,\alpha(\cdot),\ot}\leq \Lambda,
\end{align*} then the following inequality
\begin{equation*}
|u|_{2,\alpha(\cdot), \Omega_T}\leq C\left(|u|_{0, \ot}+|f|_{0,\alpha(\cdot), \ot}+|\varphi|_{2,\ac,\ot}\right),
\end{equation*}
is satisfied and $C=C(\textup{diam}(\Omega), T, n, \alpha^-, \alpha^+, c_{\log}(\alpha), \Lambda, \lambda)$.
\end{theo}
\begin{proof}
First, let us note that we can take $\varphi=0$, because the equation is linear.

Let $\delta>0$ be such as in Lemma \ref{brzglscheslem}. 
Let $P,Q\in\ot$ be arbitrary points. We consider three cases.
\begin{enumerate}
\item If $P,Q\in B(X,\delta)$ for certain $X\in \gt$, then by Lemma \ref{brzglscheslem}  we get
\begin{align}\label{in1902}
\frac{|D^2u(P)-D^2u(Q)|}{d^{\alpha(P}(P,Q)}\leq C\left(|u|_{0,\ot}+|f|_{0,\ac,\ot}\right).
\end{align}
\item When $d_{P,Q}>\delta/2$, then we apply Theorem \ref{sches}. Indeed, we have
\begin{align*}
d_{P,Q}^{\alpha(P)}\frac{|D^2u(P)-D^2u(Q)|}{d^{\alpha(P)}(P,Q)}\leq C\left(|u|_{0,\ot}+|f|^{(2)}_{0,\ac,\ot}\right)\leq \left(|u|_{0,\ot}+|f|_{0,\ac,\ot}\right).
\end{align*}
We estimate from below $d_{P,Q}$, what yields
\begin{align*}
\textnormal{min}\left\{\left(\frac{\delta}{2}\right)^{\alpha^+},\left(\frac{\delta}{2}\right)^{\alpha^-}\right\}\frac{|D^2u(P)-D^2u(Q)|}{d^{\alpha(P)}(P,Q)}\leq C\left(|u|_{0,\ot}+|f|_{0,\ac,\ot}\right).
\end{align*}
\item If $d_P<\delta/2$ and there does not exist $X\in\gt$ such that $P,Q\in B(X,\delta)$. Then, there exists $P\in\gt$ such that $P\in B(X,\delta/2)$. Thus, we have
\begin{align*}
d(P,Q)\leq d(Q,X)-d(X,P)>\delta-\delta/2=\delta/2.
\end{align*}
Hence, we obtain
\begin{align*}
\frac{|D^2u(P)-D^2u(Q)|}{d^{\alpha(P}(P,Q)}\leq \textnormal{max}\left\{\left(\frac{2}{\delta}\right)^{\alpha^-},\left(\frac{2}{\delta}\right)^{\alpha^+}\right\}\left(|D^2u(P)|+|D^2u(Q)|\right).
\end{align*}
We again use Lemma \ref{brzglscheslem} and Theorem \ref{sches} and we finally get the same  inequality as in (\ref{in1902}).
\end{enumerate}
\end{proof}
Next theorem is a simple consequence of the the previous one. 
\begin{theo}\label{mainl}
Let us assume that $\Omega$ is a set of class $C^{2,\alpha^+}$. 
If $u\in \chp{2}{1}(\otb)$ satisfies
\begin{align*}
u_t-Lu={}&f\textrm{ on }\ot,\\
u={}& \varphi\textrm{ on }\gt,
\end{align*}
where $f\in C^{\alpha(\cdot)}(\otb)$, $\varphi\in C^{2,\ac}(\otb)$ and there are positive constants $\lambda$ and $\Lambda$ such that
\begin{align*}
&a^{ij}(x,t)\zeta^i\zeta^j\geq\lambda|\zeta|^2\qquad\textrm{for all $x\in\Omega$ and $0\leq t<T$ and for all $\zeta\in\mathbb{R}^n$,}\\
&|a^{ij}|_{0,\alpha(\cdot),\ot},|b^i|_{0,\alpha(\cdot),\ot},|c|_{0,\alpha(\cdot),\ot}\leq \Lambda,
\end{align*} then the following inequality
\begin{equation*}
|u|_{2,1,\alpha(\cdot), \Omega_T}\leq C\left(|u|_{0, \ot}+|f|_{0,\alpha(\cdot), \ot}+|\varphi|_{2,\ac,\ot}\right),
\end{equation*}
is satisfied and $C=C(\textup{diam}(\Omega), T, n, \alpha^-, \alpha^+, c_{\log}(\alpha), \Lambda, \lambda)$.
\end{theo}

\section{Existence of solutions}\label{molsec}

In this section we shall prove the following Kellogg's type theorem.
\begin{theo}\label{finth}
Let $\Omega\subset\mathbb{R}^n$ be an open and bounded set with the boundary of class $C^{2,\alpha^+}$ and $T>0$.
Let $u_t-Lu=u_t-\left(a^{ij}D_{ij}u+b^i D_iu+cu\right)$ be a operator satisfying
\begin{align*}
&a^{ij}(x,t)\zeta^i\zeta^j\geq\lambda|\zeta|^2\qquad\textrm{for all $x\in\Omega$ and $0\leq t<T$ and for all $\zeta\in\mathbb{R}^n$,}
\end{align*} with 
coefficients in $C^{\alpha(\cdot)}(\bar{\Omega})$ and  $c\geq 0$. If $f\in C^{\alpha(\cdot)}(\otc)$, $\phi\in \chp{2}{1}(\otc)$ and $\phi_t-L\phi=f$ on $\partial\Omega\times\{0\}$,
then the problem 
\begin{align*}
\begin{cases}
u_t-L u=f & \text{in} \ \ot,\\
u=\phi & \text{on} \ \gt,
\end{cases}
\end{align*}
 has a unique solution $u\in \chp{2}{1}(\otc)$.
\end{theo}

 First, we prove an extension lemma for H\"older functions. For given $\ot$ and $\sigma>0$ we define a set
\begin{align*}
\ots=\left\{X\in\R^{n+1}\colon \dist{X}{\ot}<\sigma\right\},
\end{align*}
where the distance is calculated in the metric defined in (\ref{metr}).
\begin{lem}\label{extlem}
Let $\Omega\subset\mathbb{R}^n$ be an open and bounded set with the boundary of class $C^2$ and let $T>0$. 
Then, there exists $\sigma>0$ such that there exists $\bar{\alpha}\in\mathcal{A}^{\log}(\ots)$  with $\bar{\alpha}|_{\ot}=\alpha$, $\bar{\alpha}^+=\alpha^+$ ,
 $\bar{\alpha}^-=\alpha^-$ such that for any $f\in C^{\alpha({\cdot})}(\otc)$, there exists 
$\bar{f}\in C^{\bar{\alpha}(\cdot)}(\otsc)$ satisfying $\bar{f}|_{\ot}=f$. Moreover, there exists a constant $C=C(\Omega, n, \alpha^-, \alpha^+, c_{\log}(\alpha),T)$ such that the inequality
\begin{align}\label{extin}
|\bar{f}|_{0,\bar{\alpha}(\cdot),\ots}\leq C|f|_{0,\alpha(\cdot),\ot}
\end{align}
holds and $C=C(\Omega,n,T,c_{\log}(\alpha), \alpha^+,\alpha^-)$.
\end{lem}
\begin{proof}
First, we extend $f$ to a set $\Omega\times(-\sigma^2,T+\sigma^2)$. Let $\tilde{f}\colon \Omega\times(-\sigma,T+\sigma)\to \R$ be defined as follows
\begin{align*}
\tilde{f}(x,t)=\begin{cases}
f(x,t),&\textrm{if }0< t< T,\\
f(x,T),&\textrm{if } T\leq t< T+\sigma^2,\\
f(x,0),&\textrm{if }-\sigma^2 < t\leq 0.
\end{cases}
\end{align*}
Similarly, we define
\begin{align*}
\tilde{\alpha}(x,t)=\begin{cases}
\alpha(x,t),&\textrm{if }0< t< T,\\
\alpha(x,T),&\textrm{if } T\leq t< T+\sigma^2,\\
\alpha(x,0),&\textrm{if }-\sigma^2 < t\leq 0.
\end{cases}
\end{align*}
It is easy to check that $\tilde{f} \in C^{\tilde{\alpha}(\cdot)}(\otsc)$. The positive $\sigma$ will be specified later.

Now, let us take $t_0\in [-\sigma^2,T+\sigma^2]$. We can extend the function $\tilde{f}(\cdot,t_0)$  to $\Omega_{\sigma}$ using Lemma \ref{extlemsch}. We choose $\sigma$ as in this lemma. We proceed with $\tilde{\alpha}$ analogously. Thus, we have $\bar{\alpha}$ and $\bar{f}$, defined on $\ots$. We will show that $\bar{\alpha}\in \mathcal{A}^{\log}(\ots)$. Let
\begin{align*}
\cdot^*\colon\left(\partial\Omega\right)_{\sigma}\to\left(\partial\Omega\right)_{\sigma},
\end{align*}
be the mapping from the proof of Lemma \ref{extlemsch} in \cite{bies1}. We represent each $x\in\left(\partial\Omega\right)_{\sigma}$ as follows 
\begin{align*}
x=x_0+dn(x_0),
\end{align*}
where $x_0\in \partial\Omega$, $n(x_0)$ is an exterior unit normal vector and $d\in (-\sigma,\sigma)$. The point $x_0$ and the number $d$ are uniquely determined. For $x=x_0+dn(x_0)$ define $$x^*=x_0-dn(x_0).$$
Therefore, 
\begin{align*}
\bar{f}(x,t)=\tilde{f}(x^*,t),\qquad \bar{\alpha}(x,t)=\tilde{\alpha}(x^*,t)\qquad\textrm{for $x\in\Omega_{\sigma}\setminus\Omega$ and $-\sigma^2\leq t\leq T+\sigma^2$.}
\end{align*}

 Let us take $P=(x,t),Q=(y,s)\in\ots$ such that $d(P,Q)\leq \tfrac{1}{2}$.
We estimate 
\begin{align*}
\left|\ln (d(P,Q))\right|\left|\bar{\alpha}(P)-\bar{\alpha}(Q)\right|&\leq \left|\ln d(P,Q)\right|\left(|\bar{\alpha}(P)-\bar{\alpha}(y,t)|+|\bar{\alpha}(y,t)-\bar{\alpha}(Q)|\right)\\
&\leq \left|\ln(|x-y|)\right|\bar{\alpha}(P)-\bar{\alpha}(y,t)|+\big|\ln|s-t|^{1/2}\big||\bar{\alpha}(y,t)-\bar{\alpha}(Q)|=W_1+W_2.
\end{align*}
Clearly
\begin{align*}
W_1\leq c_{\log}(\alpha).
\end{align*}
We have to consider two cases to estimate $W_2$.
\begin{enumerate}
\item If $y\in \Omega$, then
\begin{align*}
W_2\leq c_{\log}(\alpha).
\end{align*}
\item If $y\in(\partial\Omega)_{\sigma}\setminus\Omega$, then we write
\begin{align*}
W_2=\big|\ln|s-t|^{1/2}\big||\bar{\alpha}(y^*,t)-\bar{\alpha}(y^*,s)|\leq c_{\log}(\alpha).
\end{align*}
\end{enumerate}
Next, if $d(P,Q)>\frac{1}{2}$, we estimate as follows 
\begin{align*}
|\bar{\alpha}(P)-\bar{\alpha}(Q)||\ln d(P,Q)|\leq 2\alpha^+\max\left(\left|\ln\diam{\ot}\right|, \ln 2 \right).
\end{align*}
Thus, we see that $\bar{\alpha}\in \mathcal{A}^{\log}(\ots)$.

It is left to check that $\bar{f}\in C^{\bar{\alpha}(\cdot)}(\otsc)$. For $P=(x,t),Q=(y,s)\in\ots$ we proceed as follows
\begin{align*}
|\bar{f}(P)-\bar{f}(Q)|\leq |\bar{f}(P)-\bar{f}(y,t)|+|\bar{f}(y,t)-\bar{f}(Q)|.
\end{align*}
Terms on the right--hand side can be estimated in the similar way as $\bar{\alpha}$.
\end{proof}

We define a ball centered at $x\in\R^n\times\R$ with radius $r>0$ as usual 
$$
B(X,r)=\left\{Y\in\R^n\times\R\colon |X-Y|<r\right\}.
$$
We used there the Euclidean norm.

Now, let $\phi$ be the standard mollifier, i.e. $\phi\geq 0$, $\textnormal{supp}\,\phi\subset B(0,1)$, $\int_{B(0,1)}\phi\dx x=1$ and $\phi$ is smooth. For $\epsilon>0$ we denote $\phi_{\epsilon}(\cdot)=\frac{1}{\epsilon^n}\phi(\frac{\cdot}{\epsilon})$. Then, if $f\in L^1_{\textnormal{loc}}$, we define $f_{\epsilon}=f*\phi_{\epsilon}$.
\begin{lem}\label{molem}
Let $\Omega\subset\R^n$ be an open and bounded set and fix $T>0$, $\sigma>0$. Then for any $\delta\in(0,\alpha^-)$ there exists $\epsilon'=\epsilon'(\delta)>0$ such that for all $\epsilon\leq\epsilon'$, $f\in\chpbt(\otsc)$ the following inequality
\begin{equation*}
|f_{\epsilon}|_{0,\ac-\delta,\ot}\leq 3|f|_{0,\ac,\ots}
\end{equation*}
holds.
\end{lem}
\begin{proof}
Because $\alpha$ is log-H\"older continuous, so it is also uniformly continuous. Therefore, there exists $\epsilon'>0$ such that for all $X,Y\in\ots$ with $|X-Y|<\epsilon'$ we have
\begin{align*}
|\alpha(X)-\alpha(Y)|<\delta.
\end{align*}
Consequently
\begin{align}\label{pomin2011}
\alpha(X)-\alpha(Y)>-\delta.
\end{align}
Fix $X,Y\in\ots$ with $d(X,Y)<1$, then
\begin{align}\label{in21111}
\frac{|f_{\epsilon}(X)-f_{\epsilon}(Y)|}{d^{\alpha(X)}(X,Y)}&\leq\int_{B(0,\epsilon)}\phi_{\epsilon}(Z)\frac{|f_{\epsilon}(X-Z)-f_{\epsilon}(Y-Z)|}{d^{\alpha(X)}(X,Y)}\dx Z
\nonumber\\&\leq [f]_{0,\ac,\ot} \int_{B(0,\epsilon)}\phi_{\epsilon}(Z)d^{\alpha(X-Z)-\alpha(X)}(X,Y)\dx Z.
\end{align}
Since $d(X,Y)<1$ and $d(X-Z,X)<\epsilon$, by (\ref{pomin2011}) we have 
\begin{align*}
d^{\alpha(X-Z)-\alpha(X)}(X,Y)\leq d^{-\delta}(X,Y).
\end{align*}
Thus, from (\ref{in21111}) we get
\begin{align*}
\frac{|f_{\epsilon}(X)-f_{\epsilon}(Y)|}{d^{\alpha(X)-\delta}(X,Y)}\leq [f]_{0,\ac,\ots}.
\end{align*}
Now, let us take $X,Y\in \ot$ such that $d(X,Y)\geq 1$. Then 
\begin{align*}
\frac{|f_{\epsilon}(X)-f_{\epsilon}(Y)|}{d^{\alpha(X)+\delta}(X,Y)}\leq 2|f_{\epsilon}|_{0,\ot}\leq 2|f|_{0,\ots}.
\end{align*}
Thus, finally the above inequality yields
\begin{align*}
[f_{\epsilon}]_{0,\ac-\delta,\ot}\leq 2|f_{\epsilon}|_{0,\ot}\leq 2|f|_{0,\ots}+[f]_{0,\ac,\ots}.
\end{align*}
Hence,
\begin{align*}
|f_{\epsilon}|_{0,\ac-\delta,\ot}\leq 3|f|_{0,\ac,\ots}.
\end{align*}
\end{proof}
Now, we are able to prove Theorem \ref{finth}.
\begin{proof}[Proof of Theorem \ref{finth}]
Because the operator $u_t-Lu$ is linear, we can assume that boundary values are equal to zero i.e. $\phi=0$. First, we consider the following nonhomogeneous heat equation. That is the equation of the form
\begin{align*}
u_t-Lu=u_t-\Delta u=f.
\end{align*}
Next, we apply Lemma \ref{extlem}. Let $\sigma$ be as in this lemma and let $\bar{f}$ and $\bar{\alpha}$ be extensions to $\ots$ of $f$ and $\alpha$ respectively . Then, we can mollify the function $\bar{f}$ on $\ot$. From Lemma \ref{molem}, we get the following two sequences
\begin{align*}
\delta_m\to 0^+,\qquad \epsilon_m\to 0^+.
\end{align*}
We can assume that both of them are decreasing. For all $m\in \mathbb{N}$ the inequality
\begin{align*}
|\bar{f}_{\epsilon_m}|_{0,\ac-\delta_m,\ot}\leq 3|\bar{f}|_{0,\ac, \ots}
\end{align*}
holds. Without loss of generality, we can assume that for all $m$ the inequality $\delta_m<\alpha^-$ is satisfied.

Inequality (\ref{extin}) yields 
\begin{align}\label{in1e}
|\bar{f}_{\epsilon_m}|_{0,\ac-\delta_m,\ot}\leq 3|{f}|_{0,\ac, \ots}.
\end{align} 
 Since $\bar{f}_{\epsilon_m}\in C^{\infty}(\otc)$, we can  use the theory of existence for H\"older spaces with a constant variable.  Thus, the problem
\begin{align*}
u_{\epsilon_m,t}-Lu_{\epsilon_m}={}&\bar{f}_{\epsilon_m}\textrm{ on }\ot,\\
u_{\epsilon_m}={}& 0\textrm{ on }\gt
\end{align*}
has got a unique solution $u_{\epsilon_m}\in C^{2,1,\alpha^+-\delta_m}(\otc)$ for all $m$.

Using inequality from Theorem \ref{mainl} and inequality (\ref{in1e}) we obtain
\begin{align}\label{in6091}
|u_{\epsilon_m}|_{2,1,\ac-\delta_m,\ot}\leq C\left(|\bar{f}_{\epsilon_m}|_{0,\ac-\delta_m,\ot}+|u_{\epsilon_m}|_{0,\ot}\right)\leq C\left(|f|_{0,\ac,\ot}+|u_{\epsilon_m}|_{0,\ot}\right).
\end{align}
The constant from Theorem \ref{mainl} depends on $\alpha^+-\delta_m$ and $\alpha^--\delta_m$. Therefore, it depends on $m$, but it can be shown that we can take a finite $C$, which is good for all $m$.

In virtue of the maximum principle, we conclude that $|u_{\epsilon_m}|_{0,\ot}\leq C |f_{\epsilon_m}|_{0,\ot}\leq |f|_{0,\ot}$. Therefore, (\ref{in6091}) implies
\begin{align}\label{in2711}
|u_{\epsilon_m}|_{2,1,\ac-\delta_m,\ot}\leq C |f|_{0,\ac,\ot}.
\end{align}
Let us take $\gamma$ such that $\alpha^--\delta_m\geq \gamma>0$. Then, from (\ref{in2711}) we obtain  
\begin{align*}
|u_{\epsilon_m}|_{2,1,\gamma,\ot}\leq C |f|_{0,\ac,\ot}.
\end{align*}
Hence, the sequence $\{u_{\epsilon_m}\}$ is bounded in the space $C^{2,1,\gamma}(\otc)$. Therefore, by the Arzela--Ascoli Theorem, we conclude that there exists a subsequence, still denoted as $u_{\epsilon_m}$, and $u\in C^{2,1,\gamma}(\otc)$ such that 
\begin{align*}
u_{\epsilon_m}\to u\textrm{\quad in } C^{2,1}(\otc).
\end{align*}

Letting $m\to \infty$ in $u_{\epsilon_m,t}-\Delta u_{\epsilon_m}=\bar{f}_{\epsilon_m}$, we obtain
\begin{align*}
u_t-\Delta u=f.
\end{align*}
Moreover, by (\ref{in2711}) there exists $M>0$ such that for all $X,Y\in \ot$ we have
\begin{align*}
\frac{|D^2_xu(X)-D^2_xu(Y)|}{d^{\alpha(X)-\delta_m}(X,Y)}+\frac{|u_t(X)-u_t(Y)|}{d^{\alpha(X)-\delta_m}(X,Y)}\leq M.
\end{align*}
Letting $m\to\infty$ in the above inequality, we conclude that $u\in C^{2,1,\ac}(\ot)$. This ends the proof for the heat equation.

Now, let $\tfrac{\partial}{\partial t}-L$ be an arbitrary parabolic operator. We will apply the method of continuity in this case. Let $T_0=\frac{\partial}{\partial t}-\Delta$ and $T_1=\frac{\partial}{\partial t}-L$. We define an operator $T_s=(1-s)T_0+sT_1$ for $s\in [0,1]$. By Theorem~\ref{mainl}, we obtain that there exists a constant C such that for all $s\in [0,1]$ and $u\in C^{2,1\ac}(\otc)$ the following inequality
\begin{align*}
|u|_{2,1,\ac,\ot}\leq C\left(|u|_{0,\ot}+|T_su|_{0,\ac,\ot}\right)
\end{align*}
is satisfied. Since $c\geq 0$, the maximum principle implies $|u|_{0,\ot}\leq C |T_su|_{0,\ot}$. Combining this with the above inequality yields 
\begin{align*}
|u|_{2,1,\ac,\ot}\leq C|T_su|_{0,\ac,\ot}.
\end{align*}
Finally, thanks to the method of continuity, we obtain existence of solutions in a general case.
\end{proof}

We finish the article with the following example.
\begin{ex}
Let us fix $e^{-2}<\gamma<1$ and $\zeta<1-\gamma$. Set $\Omega=B(0,\zeta)$ and $T=\zeta$. Moreover, let $\alpha\colon\ot\to(0,1)$ be a variable exponent defined as follows
\begin{align*}
\alpha(x,t)=(\gamma+|x|)(\gamma+t).
\end{align*}
It is easy to see that
\begin{align*}
\alpha^+=(\gamma+\zeta)^2,\quad\alpha^-=\gamma^2.
\end{align*}
We claim that the exponent $\alpha$ is log-H\"older continuous. Indeed, for $(x,t),(y,s)\in\ot$ we have 
\begin{align*}
|\alpha(x,t)-\alpha(y,s)|\left|\ln d\left((x,t),(y,s)\right)\right|&\leq \left(|\alpha(x,t)-\alpha(y,t)|+|\alpha(y,t)-\alpha(y,s)|\right)\left|\ln d\left((x,t),(y,s)\right)\right|\\
&\leq \left(|x-y|+|s-t|\right)\left|\ln d\left((x,t),(y,s)\right)\right|.
\end{align*}
The right-hand side of the above inequality is bounded, so $\alpha$ is log-H\"older continuous.

We define $f\colon\ot\to\R$ as follows
\begin{align*}
f(x,t)=(|x|+\sqrt{t})^{\alpha(x,t)}\textrm{ for }(x,t)\in\ot.
\end{align*}
One can check that 
\begin{align*}
|f(x,t)-f(y,s)|\leq C d^{\alpha(x,t)}((x,t),(y,s)) \textrm{ for }(x,t),(y,s)\in\ot.
\end{align*}
Thus, $f\in C^{\ac}(\otc)$.

We consider the problem
\begin{align}\label{rowex}
\begin{cases}
L u=f & \text{in} \ \ot,\\
u=0 & \text{ on} \ \gt.
\end{cases}
\end{align}
Due to Theorem \ref{finth}, it has a unique solution $u\in \chp{2}{1}(\otc)$. Notice that $f\notin C^{\beta}(\otc)$, where $\beta$ is a constant exponent and $\beta\in (\alpha^-,1)$. Indeed, to see this take a sequence $\theta_n=(\zeta/n,0,\ldots,0,1/n^2)\in\ot$ and a point $\theta_0=(0,\ldots,0)$. Then 
\begin{align*}
\frac{|f(\theta_n)-f(0)|}{|\theta_n|^{\beta}}\geq (\zeta+1)^{\alpha^--\beta}\frac{}{}n^{\beta-(\gamma+\zeta/n)(\gamma+1/n)}\longrightarrow\infty.
\end{align*}
Thus,(\ref{rowex}) has no solutions in the space $C^{2,1,\beta}(\otc)$ for $\beta\in (\alpha^-,1)$.
\end{ex}

\section*{Acknowledgments}
The author is supported by National Science Centre, Poland Grant No. 2016/21/N/ST1/01389. P.M.B. wishes to express his thanks to Tomasz Kostrzewa for pointing out mistakes in English and in \LaTeX. He is also greatly indebted to Henning Kempka from University of Applied Sciences for stimulating conversations about mathematics and for arranging Bies' visitation on University of Jena, where a part of the article was written. The author would like to thank Katarzyna Maria D\'zwiga\l a for checking English and acknowledges the many helpful suggestions of Przemys\l aw G\'orka during the preparation of the paper too.
\begin{appendices}
\section{Interpolation Inequalities}\label{apint}
We present here the interpolation inequalities for parabolic H\"older spaces with variable exponent.
 Proofs of these results are almost the same as proofs of analogous facts in Appendix A in \cite{bies1}. Thus, we only formulate  these lemmata without proofs.
\begin{lem}\label{intlem}
Let us assume that $\Omega\subset\R^n$ is an open and bounded set and $T>0$.
Let $\alpha$ and $\beta$ be variable exponents on $\ot$ such that $j+\beta^+<k+\alpha^-$, where $k$ and $j$ are non-negative integer numbers. Then for  $\epsilon>0$ there exists a constant $C$, that for $u\in C^{k,\ac}(\otb)$ the following inequalities
\begin{align*}
[u]^*_{j,\beta(\cdot),\ot}\leq C|u|_{0,\ot}+\epsilon[u]^*_{k,\ac,\ot},\\
|u|^*_{j,\beta(\cdot),\ot}\leq C|u|_{0,\ot}+\epsilon[u]^*_{k,\beta(\cdot),\ot}
\end{align*}
are satisfied and $C=C(\ot,\beta^+,\beta^-,c_{\log}(\beta), \alpha^+,\alpha^-,c_{\log}(\alpha),\epsilon,k,j)$.
\end{lem}
\begin{lem}
Let us assume that $\Omega\subset\R^n$ is an open and bounded set and $T>0$. Let us also require that $\Gamma$ is a subset of a set $\gt$ such that $\Gamma\subset\{x_n=0\}\cup \{t=0\}$.
Let $\alpha$ and $\beta$ be variable exponents on $\ot$ such that $j+\beta^+<k+\alpha^-$, where $k$ and $j$ are  non-negative integer numbers. Then for  $\epsilon>0$ there exists a constant $C$, that for  $u\in C^{k,\ac}(\otb)$ the following inequalities
\begin{align*}
[u]^*_{j,\beta(\cdot),\otg}\leq C|u|_{0,\ot}+\epsilon[u]^*_{k,\ac,\otg},\\
|u|^*_{j,\beta(\cdot),\otg}\leq C|u|_{0,\ot}+\epsilon[u]^*_{k,\beta(\cdot),\otg}
\end{align*}
are satisfied and $C=C(\ot,\beta^+,\beta^-,c_{\log}(\beta), \alpha^+,\alpha^-,c_{\log}(\alpha),\epsilon,k,j,\Gamma)$.
\end{lem}
\begin{lem}
Let $\Omega\subset\R^n$ be an open and bounded set.
Let $\alpha$ and $\beta$ be variable exponents on $\ot$   such that $j+\beta^+<k+\alpha^-$, where $k$ and $j$ are  non-negative integer numbers. If a boundary $\partial\Omega$ is of class $C^{k,\alpha^+}$, then for  $\epsilon>0$ there exists a constant $C$, that for  $u\in C^{k,\ac}(\otb)$ the following inequalities
\begin{align*}
[u]_{j,\beta(\cdot),\ot}\leq C|u|_{0,\ot}+\epsilon[u]_{k,\ac,\ot},\\
|u|_{j,\beta(\cdot),\ot}\leq C|u|_{0,\ot}+\epsilon[u]_{k,\beta(\cdot),\ot}
\end{align*}
are satisfied and $C=C(\ot,\beta^+,\beta^-,c_{\log}(\beta), \alpha^+,\alpha^-,c_{\log}(\alpha),\epsilon,k,j)$.
\end{lem}
\section{Estimation of integral transforms}\label{apest}

In this section we will consider functions defined on $U=\I\times (0,s_0)$, where $0<s_0\leq 1$ and $\I=(-1,1)\times\ldots\times (-1,1)\times(-1,\beta)\subset \R^n$ with $\beta\in(0,1]$. Moreover, let us denote fences
\begin{align*}
\I^{\pm}_i=\left\{x\in\overline{\I}\colon x^i=\pm 1\right\}\quad\textrm{for $i=1,\ldots,n-1$}\quad{\rm and}\quad \I^-_n=\left\{x\in\overline{\I}\colon x^n=-1\right\}, \quad\I^+_n=\left\{x\in\overline{\I}\colon x^n=\beta\right\}.
\end{align*}
 
\begin{lem}\label{lemosz}
Let us assume that $f\in C^{\ac}(\overline{U})$.
Let $N\subset U$ be such that 
\begin{align*}
\dist{N}{\I^{\pm}_i}>0
\end{align*}
for $i=1,\ldots,n-1$ and 
\begin{align*}
\dist{N}{\I^-_n}>0.
\end{align*}
In addition, we assume $\dist{N}{\mathcal{I}\times\{0\}}>0$ and  $\supp (f)\subset N$.

Let 
\begin{align*}
v(y,s)=\int_0^s\int_{\I}f(x,t)G(x,t;y,s)\dx x\dx t\textrm{ for }(y,s)\in U,
\end{align*}
then $v_s\in C^{\ac}({\overline{U}})$ and the following inequality
\begin{align*}
&\frac{\left|v_s(y_1,s_1)-v_s(y_2,s_2)\right|}{d^{\alpha(y_1,s_1)}\left((y_1,s_1),(y_2,s_2)\right)} \\&\qquad\leq
C\left([f]_{(y_1,s_1),\alpha(y_1,s_1),U}+[f]_{(y_1,s_2),\alpha(y_1,s_2),U}+[f]_{(y_2,s_1),\alpha(y_2,s_1),U}+[f]_{(y_2,s_2),\alpha(y_2,s_2),U}\right)
\end{align*}
is satisfied for $(y_1,s_1),(y_2,s_2)\in U$ and $C=C\left(\alpha^+,\alpha^-,c_{\log}(\alpha)\right)$.
\end{lem}
\begin{proof}
Let us recall the definition of $G$
$$
G(x-y,t-s)=G(x,t;y,s)=\frac{(s-t)^{-n/2}}{(2\sqrt{\pi})^n}\expo{-\frac{|x-y|^2}{4(s-t)}}.
$$
We can extend $f$ on set $\mathcal{I}\times (-\infty,0)$ by $0$. Therefore, we treat $f$ as an extension.
 
First, we consider the case when $s=s_1=s_2$ and let us take arbitrary $y_1,y_2\in \I$. We denote $P=(y_1,s)$ and $Q=(y_2,s)$.
We have
\begin{align}\label{eq571i}
v_s(y_1,s)=f(y_1,s)+\int_0^s\int_{\I}f(x,t)G_s(x,t;y_1,s)\dx x\dx t=f(y_1,s)-\int_{-\infty}^s\int_{\I}f(x,t)G_t(x,t;y_1,s)\dx x\dx t.
\end{align}
The function $G(\cdot,\cdot;y_1,s)$ satisfies the equation
\begin{align*}
G_t(\cdot,\cdot;y_1,s)+\Delta_xG(\cdot,\cdot;y_1,s)=0,
\end{align*}
so from (\ref{eq571i}) we have
\begin{align}\label{1503in1}
v_s(y_1,s)={}&f(y_1,s)-\int_{-\infty}^s\int_{\I}f(x,t)G_t(x,t;y_1,s)\dx x\dx t\nonumber\\
={}&f(y_1,s)-f(y_1,s)\int_{-\infty}^s\int_{\I}G_t(x,t;y_1,s)\dx x\dx t\nonumber\\&+\int_{-\infty}^s\int_{\I}\left(f(x,t)-f(y_1,s)\right)\Delta_yG(x,t;y_1,s)\dx x\dx t.
\end{align}
Since  
$\int^s_{-\infty}\int_{\I}G_t(x,t;y_1,s)\dx x\dx t=1$
 equality (\ref{1503in1}) we can rewrite as follows
\begin{align*}
v_s(y_1,s)=\int_{-\infty}^s\int_{\I}\left(f(x,t)-f(y_1,s)\right)\Delta_yG(x,t;y_1,s)\dx x\dx t.
\end{align*}
The same equality we have for $y_2$.

Let $\gamma=\left|y_1-y_2\right|^2$. We  write
\begin{align*}
v_s(y_1,s)-v_s(y_2,s)=I_1-I_2+I_3-I_4\left(f(y_1,s)-f(y_2,s)\right),
\end{align*}
where
\begin{align*}
I_1={}&I_1(y_1,s)=\int_{s-\gamma}^s\int_{\I}\Delta_yG(x,t;y_1,s)\left(f(x,t)-f(y_1,s)\right)\dx x\dx t,\nonumber\\
I_2={}&I_1(y_2,s)=\int_{s-\gamma}^s\int_{\I}\Delta_yG(x,t;y_2,s)\left(f(x,t)-f(y_2,s)\right)\dx x\dx t,\nonumber\\
I_3={}&\int^{{s-\gamma}}_{-\infty}\int_{\I}\left[\Delta_yG(x,t;y_1,s)-\Delta_yG(x,t;y_2,s)\right]\left(f(x,t)-f(y_1,s)\right)\dx x\dx t,\nonumber\\
I_4={}&\int^{{s-\gamma}}_{-\infty}\int_{\I}\Delta_yG(x,t;y_2,s)\dx x\dx t.
\end{align*}
We shall bound $\I_1,\ldots,\I_4$.
Let us start with $I_1$. By inequality (\ref{helpineq1}) and the fact that $f$ is H\"older continuous we obtain
\begin{align*}
|I_1|
\leq C[f]_{\alpha(P),P,U}\int_{s-\gamma}^s\int_{\I}(s-t)^{-(n+2)/2}\expo{-\frac{|x-y_1|^2}{5(s-t)}}\left(|x-y_1|+(s-t)^{1/2}\right)^{\alpha(y_1,s)}\dx x\dx t.
\end{align*}
Next, we substitute $|x-y_1|=(s-t)^{1/2}\rho$ and we get
\begin{align*}
|I_1|\leq {}&C[f]_{\alpha(P),P,U}\int^s_{s-\gamma}\int^{\infty}_0(s-t)^{-1}(s-t)^{\alpha(y_1,s)/2}(1+\rho^{\alpha(y_1,s)})\rho^{n-1}\expo{-C\rho^2}\dx \rho \dx t\\
\leq {}&C[f]_{\alpha(P),P,U}\int^s_{s-\gamma}(s-t)^{\alpha(y_1,s)/2-1}\dx t\leq C[f]_{\alpha(P),P,U}\gamma^{\alpha(y_1,s)/2}=C[f]_{\alpha(P),P,U}|y_1-y_2|^{\alpha(y_1,s)}.
\end{align*}
The equivalent thing we have got with the term $I_2$
\begin{align*}
|I_2|\leq  C[f]_{\alpha(Q),Q,U}\gamma^{\alpha(y_2,s)/2}=C[f]_{\alpha(Q),Q,U}|y_1-y_2|^{\alpha(y_2,s)}.
\end{align*}

In order to estimate $I_3$, we define 
\begin{align*}
\psi(\zeta)=y_1+\zeta(y_2-y_1)\textrm{ for }0\leq\zeta\leq 1.
\end{align*}
Thus, we get
\begin{align*}
|I_3|&\leq C[f]_{\alpha(P),P,U}\int_0^{s-\gamma}\int_{\I}\int^1_0\left|D^3_yG(x,t;\psi(\zeta),s)\right||y_1-y_2|(|x-y_1|+|s-t|^{1/2})^{\alpha(y_1,s)}\dx\zeta \dx x\dx t\\
&\leq
C[f]_{\alpha(P),P,U}
\int_0^{s-\gamma}\int_{\I}\int^{1}_{0}\left|D^3_yG(x,t;\psi(\zeta),s)\right||y_1-y_2|\\&\qquad\qquad\qquad\cdot\left(|x-\psi(\zeta)|^{\alpha(y_1,s)}+|\psi(\zeta)-y_1|^{\alpha(y_1,s)}+|s-t|^{\alpha(y_1,s)/2}\right)\dx\zeta \dx x\dx t.
\end{align*}
There again we substitute $|x-\psi(\zeta)|=(s-t)^{1/2}\rho$
\begin{align*}
|I_3|\leq{}& C[f]_{\alpha(P),P,U}|y_1-y_2|\int_{-\infty}^{s-\gamma}\int^{1}_{0}(s-t)^{-3/2}\left((s-t)^{\alpha(y_1,s)/2}+|\psi(\zeta)-y_1|^{\alpha(y_1,s)}\right)\dx\zeta \dx t\\
\leq{} &C[f]_{\alpha(P),P,U}|y_1-y_2|\int^{s-\gamma}_{-\infty}(s-t)^{(\alpha(y_1,s)-3)/2}+|y_1-y_2|^{\alpha(y_1,s)}{(s-t)^{-3/2}}\dx t\\
\leq{}& C[f]_{\alpha(P),P,U}|y_1-y_2|^{\alpha(y_1,s)}.
\end{align*}
Finally we turn our attention to $I_4$. By direct calculations we have
\begin{align*}
|I_4|=\int_{-\infty}^{s-\gamma}\int_{\I}G_t(x,t;y_2,s)\dx x\dx t=\int_{\I}G(x,s-\gamma;y_2,s)\dx x\leq C.
\end{align*}
Now, if we join together all these inequalities, we shall finish the proof for this case.

Subsequently, we will consider the case $y_1=y_2=y$ and $s_1\neq s_2$. Without loss of generality we can assume $s_2<s_1$. We again denote $P=(y,s_1)$ and $Q=(y,s_2)$.
Recall the equality
\begin{align*}
v_s(y,s_i)=\int_{-\infty}^{s_i}\int_{\I}\left(f(x,t)-f(y,s_i)\right)G_s(x,t,y,s_i)\dx x\dx t,
\end{align*}
where $i=1,2$.
Let us denote $\theta=s_1-s_2$. 
We shall estimate the expression
\begin{align*}
v_s(y,s_1)-v_s(y,s_2)=J_1-J_2+J_3+\left(f(y,s_2)-f(y,s_1)\right)J_4+J_5,
\end{align*}
where
\begin{align*}
J_1={}&\int^{s_2}_{s_2-\theta}\int_{\I}(f(x,t)-f(y,s_1))G_s(x,t;y,s_1)\dx x\dx t,\\
J_2={}&\int^{s_2}_{s_2-\theta}\int_{\I}(f(x,t)-f(y,s_2))G_s(x,t;y,s_2)\dx x\dx t,\\
J_3={}&\int_{-\infty}^{s_2-\theta}\int_{\I}\left(f(x,t)-f(y,s_2)\right)\left(G_s(x,t;y,s_1)-G_s(x,t;y,s_2)\right)\dx x\dx t,\\
J_4={}&\int_{-\infty}^{s_2-\theta}\int_{\I}G_s(x,t;y,s_1)\dx x\dx t,\\
J_5={}&\int_{s_2}^{s_1}\int_{\I}\left(f(x,t)-f(y,s_1)\right)G_s(x,t;y,s_1)\dx x\dx t.
\end{align*}

We have
\begin{align*}
|J_1|\leq C[f]_{\alpha(P),P,U}\int^{s_2}_{s_2-\theta}\int_{\I}\left((s_1-t)^{\alpha(y,s_1)/2}+|x-y|^{\alpha(y,s_1)}\right)\expo{-\frac{|x-y|^2}{5(s_1-t)}}(s_1-t)^{-1-n/2}\dx x\dx t.
\end{align*}
Now, we substitute $|x-y|=(s_1-t)^{1/2}\rho$. It yields
\begin{align*}
|J_1|\leq  C[f]_{\alpha(P),P,U}\int^{s_2}_{s_2-\theta}(s_1-t)^{\alpha(y,s_1)/2-1}\dx t=C[f]_{0,\ac,U}(s_1-s_2)^{\alpha(y,s_1)/2}.
\end{align*}
In the similar way we bound the integral $J_2$ and we have
\begin{align*}
|J_2|\leq C[f]_{\alpha(Q),Q,U}(s_1-s_2)^{\alpha(y,s_1)/2}.
\end{align*}
Next, we  estimate $J_3$
\begin{align*}
|J_3|&\leq C [f]_{\alpha(P),P,U}\int_0^{s_2-\theta}\int_{\I}\int_{s_2}^{s_1}\left(|x-y|^{\alpha(y,s_1)}+(s_2-t)^{\alpha(y,s_1)/2}\right)G_{ss}(x,t;y,z)\dx z\dx x\dx t\\&
\leq C [f]_{\alpha(P),P,U}\int_0^{s_2-\theta}\int_{\I}\int_{s_2}^{s_1}\left(|x-y|^{\alpha(y,s_1)}+|s_2-z|^{\alpha(y,s_1)/2}+|z-t|^{\alpha(y,s_1)/2}\right)\\&\qquad\qquad\qquad\qquad\cdot\expo{-\frac{|x-y|^2}{5(z-t)}}(z-t)^{-n-2}\dx z\dx x\dx t.
\end{align*}
We substitute $|x-y|=(z-t)^{1/2}\rho$, what yields
\begin{align*}
|J_3|\leq {}&C[f]_{\alpha(P),P,U}\int_{0}^{s_2-\theta}\int^{s_1}_{s_2}\left(\frac{|s_2-z|^{\alpha(y,s_1)/2}}{(z-t)^2}+|z-t|^{\alpha(y,s_1)/2-2}\right)\dx z\dx t\\
\leq{}& C[f]_{\alpha(P),P,U}\int_{0}^{s_2-\theta}\int^{s_1}_{s_2}\left(\frac{|s_2-z|^{\alpha(y,s_1)/2}}{(s_2-t)^2}+|z-t|^{\alpha(y,s_1)/2-2}\right)\dx z\dx t\leq C[f]_{\alpha(P),P,U}(s_1-s_2)^{\alpha(y,s_1)/2}.
\end{align*}
The term $J_5$ we estimate in the similar way as $J_1$
\begin{align*}
|J_5|\leq C[f]_{\alpha(P),P,U}(s_1-s_2)^{\alpha(y,s_1)/2}. 
\end{align*}
It is easy to see that the term $J_4$ is bounded in the independent way of $P$ and $Q$.
Hence, we have proved Lemma for $y_1=y_2=y$ and$s_1\neq s_2$.

Let us consider now general case, i.e. $y_1$, $y_2$, $s_1$ and $s_2$ are arbitrary
\begin{align*}
|f(y_1,s_1)-f(y_2,s_2)|&\leq |f(y_1,s_1)-f(y_1,s_2)|+|f(y_1,s_2)-f(y_2,s_2)|\\
&\leq C\left([f]_{\alpha(y_1,s_1),(y_1,s_1),U}+[f]_{\alpha(y_1,s_2),(y_1,s_2),U}\right)|s_1-s_2|^{\alpha(y_1,s_1)/2}\\
&\qquad\qquad+\left([f]_{\alpha(y_1,s_2),(y_1,s_2),U}+[f]_{\alpha(y_2,s_2),(y_2,s_2),U}\right)|y_1-y_2|^{\alpha(y_2,s_2)}\\
&\leq  C\left([f]_{\alpha(y_1,s_1),(y_1,s_1),U}+[f]_{\alpha(y_1,s_2),(y_1,s_2),U}+[f]_{\alpha(y_1,s_2),(y_1,s_2),U}+[f]_{\alpha(y_2,s_2),(y_2,s_2),U}\right)\\  &\qquad \qquad\qquad \qquad\qquad\qquad\qquad\qquad\qquad\qquad\qquad\qquad\qquad\cdot d\left((y_1,s_1),(y_2,s_2)\right)^{\alpha(y_1,s_1)}.
\end{align*}
In this way we have finished the proof.
\end{proof}
\end{appendices}

\bibliography{parabolbib}
\bibliographystyle{acm}
\end{document}